\documentclass{article}%
\usepackage{amsmath}
\usepackage{amsfonts}
\usepackage{amssymb}
\usepackage{graphicx}%
\providecommand{\U}[1]{\protect\rule{.1in}{.1in}}
\newtheorem{theorem}{Theorem}[section]

\newtheorem{axiom}{Theorem}

\newtheorem{corollary}[theorem]{Corollary}

\newtheorem{lemma}[theorem]{Lemma}

\newtheorem{proposition}[theorem]{Proposition}

\newenvironment{proof}[1][Proof]{\noindent\textbf{#1.} }{\ \rule{0.5em}{0.5em}}
\begin{document}

\title{Polycyclic groups and profinite isomorphism }
\author{Dan Segal}
\maketitle



\section{Introduction}

There is no algorithm to decide whether an arbitrary finitely presented group
has nontrivial profinite completion, or equivalently, a nontrivial finite
image \cite{BW}. Within the universe of residually finite groups, however, the
opposite is true: given that the finitely presented group $G$ is residually
finite, we note that $G$ has a nontrivial finite image if and only if some
generator of $G$ is not the identity; and this is decidable by the naive
routine of simultaneously listing (a) finite images and (b) consequences of
the relations. Thus we can decide whether or not the profinite completion
$\widehat{G}$ is the trivial group.

The \emph{profinite isomorphism problem} asks: given two finitely presented
groups, do they have isomorphic profinite completions? This is of course more
demanding than the `profinite triviality problem', and is undecidable even
within the class of residually finite groups \cite{BW2}.

It has long been known that among finitely presented residually finite groups,
the polycyclic groups are particularly tractable. Around 40 years ago two
major results were established:

\begin{axiom}
\label{iso}\emph{\cite{S3} }The isomorphism problem for virtually polycyclic
groups is decidable.
\end{axiom}

\begin{axiom}
\label{genus}\emph{\cite{GPS} }The (relative) profinite genus of a virtually
polycyclic group is finite.
\end{axiom}

The former means that there exists a uniform algorithm that decides, given two
finite presentations of virtually polycyclic groups, whether the two groups
are isomorphic. (The fact that the presented groups are virtually polycyclic
has to be \emph{given}: it is in principle undecidable, by the Adian-Rabin
Theorem, \cite{LS} Chap. 4, \S 4.)

Theorem \ref{genus} asserts that the class of all virtually polycyclic groups
having a given profinite completion consists of finitely many isomorphism
classes. In fact we can remove the word \textquotedblleft
relative\textquotedblright: \emph{the class of all finitely generated
residually finite groups having the same profinite completion as a given
virtually polycyclic group (consists of virtually polycyclic groups and)
contains only finitely many non-isomorphic groups}. This follows from the
theorem together with the main result of \cite{SW}, \S 4.

The main purpose of this paper is to fill an obvious gap by establishing

\begin{axiom}
\label{profiso}The profinite isomorphism problem for virtually polycyclic
groups is decidable.
\end{axiom}

That is, there exists a uniform algorithm that decides, given two finite
presentations of virtually polycyclic groups, whether the two groups have
isomorphic profinite completions.

Two groups have isomorphic profinite completions if and only if they have the
same finite images (`isomorphism' of profinite completions is supposed to be
\emph{topological} isomorphism; for virtually polycyclic groups this is the
same as abstract isomorphism, \emph{cf.} \cite{S1} Chap. 10, Prop. 4). Thus
there is a trivial procedure that will determine the \emph{negation} of
profinite isomorphism: if $\widehat{G}\ncong\widehat{H}$ we eventually find
$m\in\mathbb{N}$ such that $G/G^{m}\ncong H/H^{m}$. However to establish a
positive result is harder. This is the opposite of the situation with the
original isomorphism problem, where an isomorphism, if one exists, will
eventually be found, while non-isomorphism is much harder to establish.

Even given Theorems \ref{iso} and \ref{profiso}, there is no known procedure
to determine the \emph{class number} of a polycyclic group, that is the
cardinality of its profinite genus. This number is related, in a messy way, to
the class number of a certain algebraic group (\emph{cf.} \cite{S1} Chapters
9, 10 and \cite{PR} Chapter 8, where the authors write \textquotedblleft...
computation of class number is a difficult, even hopeless, problem in its most
general formulation\textquotedblright\ ). It is tempting to guess that at
least one of the three algorithmic problems -- `isomorphism', `profinite
isomorphism', `class number' -- should be reducible to the conjunction of the
other two; but I don't see it.

A guiding principle in number theory is that `local' problems are supposed to
be easier than `global' problems. To prove Theorem \ref{profiso}, we use the
methodology of Theorems \ref{iso} and \ref{genus} to reduce the question to a
\emph{Diophantine problem}; the theorem then follows by the following very
general result of James Ax. Let us call a finite family of polynomial
equations in finitely many variables with coefficients in $\mathbb{Z}$ a
\emph{Diophantine system}. The system is \emph{locally solvable} if it is
solvable over $\mathbb{Z}_{p}$ for every prime $p$.

\begin{axiom}
\label{ax}\emph{\cite{A} }The local solvability of a given Diophantine system
is algorithmically decidable.
\end{axiom}

Ax remarks on the contrast between this fact and the (now known)
\emph{unsolvability} of Hilbert's 10th problem. The possibility of proving
Theorem \ref{iso} rested on the reduction of that question not to any old
Diophantine problem, but to one of a very special kind: namely an orbit
problem for an arithmetic group (see \cite{S1} Chapter 8, \S D). From this
point of view, Theorem \ref{profiso} should be seen as the `easy cousin' of
Theorem \ref{iso}: apart from the group theory (and, to be fair, some
nontrivial number theory coming from the proof of Theorem \ref{genus}), it
only needs Theorem \ref{ax}.

The special case of all three theorems \ref{iso} - \ref{profiso} relating to
$\mathcal{T}$-groups -- torsion-free finitely generated nilpotent groups -- is
discussed in Section 1 of \cite{GS3}; in that context, the group theory
question translates quickly to a corresponding arithmetical question via the
Lie algebra. The analogous reduction in the general polycyclic case is harder;
but much of the hard work has essentially been done, in the course of
establishing Theorems \ref{iso} and \ref{genus}.

Both of those theorems proceed by first embedding a virtually polycyclic group
into $\mathrm{GL}_{n}(\mathbb{Z})$ in a particularly rigid way: the question
of isomorphism between abstract groups is thereby translated into one of
conjugacy between subgroups of a linear group. The same method could be
adapted to prove Theorem \ref{profiso}, but the flexibility provided by
Theorem \ref{ax} allows us to specify an isomorphism more directly.

The secondary purpose of the paper is to present a more streamlined proof of
Theorem \ref{genus}. It was suggested in \cite{S1} Chapter 10, \S F that the
proof should in principle be about half as complicated as the one presented
there (or the original one in \cite{GPS}). It turns out that some of the
relatively elementary machinery used for Theorem \ref{profiso} indeed makes
this feasible.

\section{Strategy of the proofs}

In general, I use the conventions of \cite{S3}: a group $G$ is said to be
\emph{given} if a finite presentation of $G$ is given; an element of $G$ is
given if it is given as a word in the given generators, and a subgroup is
given if it is specified by finitely many generators, each of which is given.

For some frequently used notation see the end of this section.

\medskip

Certain polycyclic groups are particularly easy to specify; they are called
\emph{splittable,} and defined in\emph{ }\cite{S1} Chapter 7, \S B. Details
will be recalled in \S \ref{sss&l}. The splittable group $G$ is called
\emph{lattice-splittable} if its Fitting subgroup $N=\mathrm{Fit}(G)$ is a
lattice $\mathcal{T}$-group, that is, the subset $\log N$ in the Lie algebra
associated to $N$ is closed under addition (\emph{cf. }\cite{S1} Chapter 10, \S D).

The first result is a mild variation on material in \cite{S3} \S 4:

\begin{proposition}
\label{first step}There is an effective procedure that does the following:
given virtually polycyclic groups $G$ and $G^{\dag}$, it finds
lattice-splittable characteristic subgroups $G_{0}\leq G,~G_{0}^{\dag}\leq
G^{\dag}$ of finite index such that every isomorphism $\widehat{G}%
\rightarrow\widehat{G^{\dag}}$ sends $\overline{G_{0}}$ onto $\overline
{G_{0}^{\dag}}$.
\end{proposition}

The proof is recalled in \S \ref{finding}. Most of the work goes into establishing

\begin{proposition}
\label{finiteexrtrabit}Let $G_{0}\leq G,~G_{0}^{\dag}\leq G^{\dag}$ be as in
Proposition \ref{first step}, and put $N=\mathrm{Fit}(G_{0}),~N^{\dag
}=\mathrm{Fit}(G_{0}^{\dag})$. Suppose that $\theta:\widehat{G}\rightarrow
\widehat{G^{\dag}}$ an isomorphism. Then there exists $v\in\overline{G_{0}}$
\ such that $\Theta:=v^{\ast}\circ\theta$ satisfies%
\begin{equation}
(\overline{N}G)\Theta=\overline{N^{\dag}}G^{\dag}. \label{luckything}%
\end{equation}

\end{proposition}

This key result is established in \S \ref{mainproppf}. As we shall see, it is
a manifestation of the `rigidity of algebraic tori'. The point is that it
allows us to replace `exponential' Diophantine equations, involving the
entries of matrices like $\alpha^{x}$ ($\alpha$ a matrix over
$\widehat{\mathbb{Z}}$, $x$ an unknown from $\widehat{\mathbb{Z}}$), with
polynomial equations involving the entries of matrices like $\alpha^{z}$ with
$z\in\mathbb{Z}$; for each fixed choice of $z$ we obtain an ordinary
Diophantine equation with unknowns from $\widehat{\mathbb{Z}}$.

More elegant (though less immediately useful) is the following statement,
which follows from the last two results on noting that $\mathrm{Fit}%
(G_{0})\leq\mathrm{Fit}(G)$:

\begin{corollary}
\label{neatcor}Let $G$ and $G^{\dag}$ be virtually polycyclic groups with
isomorphic profinite completions. Then there exists an isomorphism
$\widehat{G}\rightarrow\widehat{G^{\dag}}$ that maps $\overline{\mathrm{Fit}%
(G)}.G$ onto $\overline{\mathrm{Fit}(G^{\dag})}.G^{\dag}$.
\end{corollary}

Keeping the above notation, set $\theta_{1}=\Theta|_{\overline{N}}$, so
$\theta_{1}:\overline{N}\rightarrow\overline{N^{\dag}}$ is an isomorphism (see
\S \ref{proftechs}). Let $\theta_{2}:G/N\rightarrow G^{\dag}/N^{\dag}$ be
induced from $\Theta$ via%
\[
G/N\cong\overline{N}G/\overline{N}\rightarrow\overline{N^{\dag}}G^{\dag
}/\overline{N^{\dag}}\cong G^{\dag}/N^{\dag}%
\]
(noting that $G\cap\overline{N}=N$ etc.). Let $\left\langle g_{1},\ldots
g_{d};~R\right\rangle $ be a finite presentation for $G,$ so $R$ is a finite
set of words $w$ such that $w(g_{1},\ldots,g_{d})=1$. Say%
\begin{equation}
Ng_{i}\theta_{2}=N^{\dag}h_{i}~(i=1,\ldots,d) \label{githeta2}%
\end{equation}
where each $h_{i}\in G^{\dag}$. For $w\in R$ we then have%
\[
b_{w}:=w(h_{1},\ldots,h_{d})\in N^{\dag}.
\]
There exist $v_{i}\in\overline{N^{\dag}}$ such that%
\[
g_{i}\Theta=v_{i}h_{i}~(i=1,\ldots,d).
\]
Now for any word $w,$%
\[
w(v_{1}h_{1},\ldots,v_{d}h_{d})=w_{\mathbf{h}}^{\prime}(v_{1},\ldots
,v_{d}).w(h_{1},\ldots,h_{d}),
\]
where the `derived word' $w_{\mathbf{h}}^{\prime}$ is a `twisted word', in
which each occurrence of $v_{i}$ in $w$ is replaced by a term of the form
$v_{i}^{x}$ where $x$ is a product of some of the $h_{j}^{\pm1}$ (\cite{S4}
\S 1.1). As $\Theta$ is a homomorphism we see that%
\begin{equation}
w\in R\Longrightarrow w_{\mathbf{h}}^{\prime}(v_{1},\ldots,v_{d})b_{w}=1.
\label{dagger-relations}%
\end{equation}

Let $\{a_{1},\ldots,a_{s}\}$ be a generating set for $N$. Then $a_{i}%
=u_{i}(\mathbf{g})$ for certain words $u_{i}$. Set $c_{i}=u_{i}(h_{1}%
,\ldots,h_{d})$, an element of $N^{\dag}$, so we have $a_{i}\Theta
=u_{i\mathbf{h}}^{\prime}(v_{1},\ldots,v_{d}).c_{i}$. Thus
\begin{equation}
a_{i}\theta_{1}=u_{i\mathbf{h}}^{\prime}(v_{1},\ldots,v_{d}).c_{i}%
~~~(i=1,\ldots,d). \label{theta1vsTheta}%
\end{equation}

Conversely, suppose we are given isomorphisms $\theta_{2}:G/N\rightarrow
G^{\dag}/N^{\dag}$ and $\theta_{1}:\overline{N}\rightarrow\overline{N^{\dag}}%
$. Choose the $h_{i}$ to satisfy (\ref{githeta2}). Now $\left\langle
g_{1},\ldots g_{d};~R\right\rangle $ is a presentation for $\widehat{G}$ in
the category of profinite groups; so if (\ref{dagger-relations}) holds then
the mapping $g_{i}\longmapsto v_{i}h_{i}~(i=1,\ldots,d)$ extends to a
homomorphism $\Theta:\widehat{G}\rightarrow\widehat{G^{\dag}}$, which induces
$\theta_{2}:G/N\rightarrow G^{\dag}/N^{\dag}$.

If in addition (\ref{theta1vsTheta}) holds, then $a_{i}\Theta=a_{i}\theta_{1}$
for each $i.$ As $\{a_{1},\ldots,a_{s}\}$ generates $\overline{N}$
topologically it follows that $\Theta$ restricts to $\theta_{1}$ on
$\overline{N}$. As $\theta_{2}$ and $\theta_{1}$ are both isomorphisms it then
follows that $\Theta:\widehat{G}\rightarrow\widehat{G^{\dag}}$ is both
surjective and injective.

\medskip We have established

\begin{proposition}
\label{mainequivalence}The following are equivalent:
\end{proposition}

\begin{description}
\item[(i)] $\widehat{G}\cong\widehat{G^{\dag}}$

\item[(ii)] There exist isomorphisms $\theta_{2}:G/N\rightarrow G^{\dag
}/N^{\dag},$ $\theta_{1}:\overline{N}\rightarrow\overline{N^{\dag}}$ and
elements $v_{1},\ldots,v_{d}\in\overline{N^{\dag}}$ such that given
(\ref{githeta2}), both (\ref{dagger-relations}) and (\ref{theta1vsTheta}) are satisfied.
\end{description}

Now isomorphisms between $\mathcal{T}$-groups, or their profinite completions,
correspond to isomorphisms between their associated Lie algebras. This means
that condition \textbf{(ii)} can be expressed in terms of linear algebra, and
doing this in an explicit manner gives

\begin{proposition}
\medskip\label{equns}There is an algorithm that does the following: given

\begin{itemize}
\item virtually polycyclic groups $G$,$~G^{\dag}$

\item Lattice $\mathcal{T}$-subgroups $N\vartriangleleft G$, $N^{\dag
}\vartriangleleft G^{\dag}$ such that every isomorphism $\widehat{G}%
\rightarrow\widehat{G^{\dag}}$ sends $\overline{N}$ onto $\overline{N^{\dag}}$

\item an isomorphism $\theta:G/N\rightarrow G^{\dag}/N^{\dag}$
\end{itemize}

it constructs a Diophantine system $\mathcal{F(\theta})$ that is locally
solvable if and only if condition \emph{(ii)} is satisfied with $\theta
_{2}=\theta$.
\end{proposition}

For details of the proof see \S \ref{dio}.

\bigskip

\textbf{Proof of Theorem \ref{profiso}.}

\medskip

Given $G$ and $G^{\dag}$ we have to decide whether or not $\widehat{G}%
\cong\widehat{G^{\dag}}$. Two procedures are run simultaneously:

\bigskip

\emph{Procedure 1}: Enumerate the finite quotients $G/G^{m},~G^{\dag}/G^{\dag
m}$ (\emph{cf}. \cite{BCRS}, Proposition 2.10). For each $m$ decide whether or
not $G/G^{m}\cong G^{\dag}/G^{\dag m}$. When the answer is `no', stop.

\medskip

\emph{Procedure 2:} Enumerate isomorphisms $\theta:G/N\rightarrow G^{\dag
}/N^{\dag}$ (\emph{cf}. \cite{BCRS}:, Proposition 2.2, Lemma 3.8). For each
one, construct the system $\mathcal{F(\theta})$ by Proposition \ref{equns} and
decide (by Theorem \ref{ax}) whether the Diophantine system $\mathcal{F(\theta
})$ is locally solvable. When the answer is `yes', stop.

\medskip

The first procedure stops if and only if $\widehat{G}\ncong\widehat{G^{\dag}%
};$ the second stops if and only if $\widehat{G}\cong\widehat{G^{\dag}}$. Thus
the profinite isomorphism of $G$ and $G^{\dag}$ is decidable.

\bigskip

\textbf{Proof of Theorem \ref{genus}.}

\bigskip

The new ingredient is the following variation on Proposition \ref{first step}:

\begin{proposition}
\label{newfirst}There exists a mapping $G\longmapsto G_{0}$ that assigns to
each virtually polycyclic group $G$ a characteristic lattice-splittable
subgroup $G_{0}$ of finite index, in such a way that if $G^{\dag}$ is any
virtually polycyclic group and $\theta:\widehat{G}\rightarrow\widehat{G^{\dag
}}$ is an isomorphism then $\overline{G_{0}}\theta=\overline{G_{0}^{\dag}}$.
\end{proposition}

This is proved in \S \ref{lastbit}, using Corollary \ref{neatcor} and some
cohomology. (The argument does \emph{not} establish that $G_{0}$ is
effectively constructible; this can in fact be done by a slightly more
elaborate method, but is not really relevant.)

The proofs in \cite{GPS} and in \cite{S1} both reduce Theorem \ref{genus} to
the main result of \cite{GS2}, which is somewhat more general than

\begin{theorem}
\emph{\cite{GS2}} Let $\mathcal{X}$ be a family of virtually soluble subgroups
of $\mathrm{GL}_{n}(\mathbb{Z})$ such that the groups $\overline{G}$
($G\in\mathcal{X}$) are all conjugate in $\mathrm{GL}_{n}^{\infty}$. Then
$\mathcal{X}$ is contained in the union of finitely many conjugacy classes of
subgroups in $\mathrm{GL}_{n}(\mathbb{Z})$.
\end{theorem}

This depends on the arithmetic theory of algebraic groups, and seems to be the
`hard core' of Theorem \ref{genus}.

Together with Proposition \ref{iso-conj} and Lemma \ref{C-choices} in
\S \ref{sss&l} below, it implies the statement of Theorem \ref{genus} for the
special case of lattice-splittable groups (namely Proposition 10 of \cite{S1},
Chapter 10).

Accepting this we can now deduce the full result; in contrast to the treatment
in \cite{S1}, the present direct argument has no recourse to the theory of
arithmetic groups. Let $\mathcal{C}$ be a family of virtually polycyclic
groups all having isomorphic profinite completions. For $G\in\mathcal{C}$ let
$G_{0}\vartriangleleft G$ be the lattice-splittable subgroup provided by
Proposition \ref{newfirst}. Then the groups $G_{0}$ for $G\in\mathcal{C}$ all
have isomorphic profinite completions, hence they lie in finitely many
isomorphism classes by the `special case'.

If $G,~G^{\dag}\in\mathcal{C}$ then
\[
\left\vert G:G_{0}\right\vert =\left\vert \widehat{G}:\overline{G_{0}%
}\right\vert =\left\vert \widehat{G^{\dag}}:\overline{G_{0}^{\dag}}\right\vert
=\left\vert G^{\dag}:G_{0}^{\dag}\right\vert ,
\]
so $\left\vert G:G_{0}\right\vert =d$, say, is constant for $G\in\mathcal{C}$.
The theorem now follows by \cite{S1}, Chapter 8, Theorem 6 which says that
given a virtually polycyclic group $G_{0}$ there are only finitely many
non-isomorphic groups $G$ having $G_{0}$ as a normal subgroup of index $d$.

\bigskip

\emph{Some Notation}

\bigskip

$\mathcal{T\qquad}$The class of torsion-free finitely generated nilpotent goups

$\overline{X}\qquad\ \ $closure of subset $X\ $in a topological group$\ $

$A\leq_{f}G\qquad A$ is a subgroup of finite index in $G$

$x^{\ast}\qquad$inner automorphism induced by $x$

$\mathrm{U}_{n}\qquad$Group of upper unitriangular $n\times n$ matrices

$\mathfrak{u}_{n}\qquad$Lie algebra of upper zero-triangular $n\times n$ matrices

$\mathrm{D}_{n}\qquad$Group of diagonal $n\times n$ matrices

$\mathrm{Sym}(n)\qquad$Symmetric group of degree $n,$ also the group of
$n\times n$ permutation matrices

$\mathrm{Hol}(G)=G\rtimes\mathrm{Aut}(G)\qquad$Holomorph of the group $G$

$\mathrm{Fit}(G)\qquad$Fitting subgroup of $G$ (maximal nilpotent normal subgroup)

$\mathrm{Z}(G),~G^{\prime},~\gamma_{n}(G)\qquad$centre, derived group, $n$th
term of lower central series of $G$

$\Gamma(G)\qquad$smallest normal subgroup $N$ such that $G/N\in\mathcal{T}$

$\tau(G)\qquad$torsion subgroup of nilpotent group $G$

$\tau_{i}(G/\gamma_{n}(G))=\tau(G/\gamma_{n}(G))$

$\mathrm{GL}_{n}^{\infty}=\mathrm{GL}_{n}(\widehat{\mathbb{Z}})=\prod
\nolimits_{p\text{ prime}}\mathrm{GL}_{n}(\mathbb{Z}_{p})$

$\mathfrak{o}_{k}\qquad$Ring of integers of number field $k$

$\pi_{\mathfrak{p}}:\mathrm{GL}_{n}(\mathfrak{o}_{k})\rightarrow
\mathrm{GL}_{n}(\mathfrak{o}_{k}/\mathfrak{p})\qquad$natural homomorphism
reduction modulo $\mathfrak{p}$, an ideal in $\mathfrak{o}_{k}$

\section{Technicalities\label{proftechs}}

This is all well known material, summarised here for convenience.

\subsection{Profinite completions\label{proftech1}}

The passage from a virtually polycyclic group $G$ to its profinite completion
$\widehat{G}$ is very well behaved. Write $\overline{X}$ to denote the closure
in $\widehat{G}$ of a subset $X$. For a subgroup $A$ and a normal subgroup $N$
of $G$ denote by $\pi:\widehat{G}\rightarrow\widehat{G/N}$ and $i:\widehat{A}%
\rightarrow\widehat{G}$ the morphisms induced by the quotient map
$G\rightarrow G/N$ and the inclusion $A\hookrightarrow G$. Let $B\leq G$. Then%
\begin{align}
\widehat{G}/\overline{N}  &  \cong\widehat{G/N}\text{ induced by }%
\pi\nonumber\\
\widehat{A}  &  \cong\overline{A}\text{ induced by }i\nonumber\\
\overline{A}\cap G  &  =A\nonumber\\
\overline{A}\cap\overline{B}  &  =\overline{A\cap B},~~\mathrm{C}%
_{\overline{B}}(\overline{A})=\overline{\mathrm{C}_{B}(A)},~~\mathrm{N}%
_{\overline{B}}(\overline{A})=\overline{\mathrm{N}_{B}(A)}~\label{subgps}\\
\mathrm{Fit}(\widehat{G})  &  =\overline{\mathrm{Fit}(G)}\label{Fit}\\
w(\widehat{G})  &  =\overline{w(G)}~~(w\text{ any word})\label{powers}\\
\tau(\widehat{G})  &  =\overline{\tau(G)}\text{ if }G\text{ is nilpotent}%
\label{tau}\\
\overline{A}\overset{\widehat{G}}{\sim}\overline{B}  &  \Longleftrightarrow
A\overset{G}{\sim}B \label{subconj}%
\end{align}
where $\overset{G}{\sim}$ means `is conjugate in $G$ to'. $\mathrm{Fit}(G)$
denotes the Fitting subgroup of $G$.

The first claim is a general fact. For the second claim see \cite{S1} Chap.
10, \S B, and for the third \emph{loc. cit.} Chap 1, Ex. 11. (\ref{subgps}) is
established in \cite{RSZ}, Props. 3.3, 3.5. For (\ref{Fit}) (Pickel's theorem)
see \cite{S1} Chap. 10, \S C; for (\ref{powers}) see \cite{S4}, \S 4.1 and
\S \ 2.3; (\ref{tau}) is an easy exercise; and (\ref{subconj}) is \cite{S1}
Chap. 4, Theorem 7.

Recall that%
\[
\mathrm{GL}_{n}(\mathbb{Z)}<\mathrm{GL}_{n}^{\infty}=\mathrm{GL}%
_{n}(\widehat{\mathbb{Z}})=\prod\limits_{p}\mathrm{GL}_{n}(\mathbb{Z}_{p}).
\]
$\mathrm{GL}_{n}^{\infty}$ is endowed with the congruence topology, induced by
the $p$-adic topology on each $\mathrm{GL}_{n}(\mathbb{Z}_{p})$. If
$G<\mathrm{GL}_{n}(\mathbb{Z)}$ is virtually polycyclic, the embedding
$G\rightarrow\overline{G}$ induces an isomorphism $\widehat{G}\rightarrow
\overline{G}$ (\cite{S1} Chap. 10, Prop. 5). This is a consequence of
\cite{S1} Chap. 4, Thm. 5, which says that the profinite topology on $G$ is
induced by the congruence topology on $\mathrm{GL}_{n}(\mathbb{Z)}$, and that
$G$ is closed in the latter topology (Formanek's Theorem). As group
homomorphisms are continuous in the profinite topology, it follows that any
homomorphism $G\rightarrow H$, for virtually polycyclic groups $G,~H\leq
\mathrm{GL}_{n}(\mathbb{Z)}$, extends uniquely to a continuous homomorphism
$\overline{G}\rightarrow\overline{H}$.

We shall often use these facts without special mention. Many more or less
elementary algorithms for working with virtually polycyclic groups are
described in \cite{BCRS}; these will also sometimes be used implicitly.

\subsection{$\mathcal{T}$-groups\bigskip}

Let $N$ be a $\mathcal{T}$-group. The `canonical embedding' $\beta_{N}$ of
$N$, described in Chapter 5 of \cite{S1}, identifies $N$ with a subgroup of
$\mathrm{U}_{n}(\mathbb{Z})$ in such a way that every automorphism of $N$ is
induced by conjugation with a matrix in $\mathrm{GL}_{n}(\mathbb{Z})$.
Specifically, $\beta_{N}$ gives the action of $N$ on a module%
\[
V(N)=\mathbb{Z}N/I_{N}\cong\mathbb{Z}^{n}%
\]
for a certain characteristic ideal $I_{N}$ of the group ring $\mathbb{Z}N$,
where $N$ acts by right multiplication on $\mathbb{Z}N$. Here $n=n(N)$ is an
invariant of $N,$ in principle easily determined from the structure of $N$
(\cite{S1} Chapter 5 \S B). We often identify $N$ with $N\beta_{N}$.

We extend $\beta_{N}$ to $\mathrm{Hol}(N)$ by making automorphisms of $N$ act
in the natural way on $V(N)$.

\begin{lemma}
\label{faithful}\bigskip Let $\Gamma=\mathrm{Aut}(N)\beta_{N}$ and write
$\widetilde{\Gamma}$ for the Zariski closure of $\Gamma$ in $\mathrm{GL}_{n}$.
Then $\mathrm{C}_{\widetilde{\Gamma}}(N\beta_{N})=1$.
\end{lemma}

\begin{proof}
The action of $\mathrm{Aut}(N)$ on $V(N)$ fixes the element $1^{\sharp
}=1+I_{N},$ so $\widetilde{\Gamma}$ is contained in the stabilizer of
$1^{\sharp}$. The result follows since the $N$-translates of $1^{\sharp}$ span
$V(N)$.
\end{proof}

\bigskip

There are mutually inverse polynomial mapppings%
\begin{align*}
\log &  :\mathrm{U}_{n}(\mathbb{Q})\rightarrow\mathfrak{u}_{n}(\mathbb{Q})\\
\exp &  :\mathfrak{u}_{n}(\mathbb{Q})\rightarrow\mathrm{U}_{n}(\mathbb{Q}),
\end{align*}
and%
\[
\mathcal{L}_{\mathbb{Q}}(N):=\mathbb{Q}\log N
\]
is a Lie subalgebra of $\mathfrak{u}_{n}(\mathbb{Q})$. The Mal'cev completion
$N^{\mathbb{Q}}$ of $N$ may be identified with $\exp\mathcal{L}_{\mathbb{Q}%
}(N)$. (Of course the Lie algebra and the mappings $\log$ and $\exp$ may be
defined intrinsically, using the Campbell-Hausdorff formula: the embedding
$\beta$ is merely a convenience. \emph{cf}. \cite{Ba} Chapter 4, \cite{Wa}
Chapter 12).

The subset $\log N\subseteq\mathcal{L}_{\mathbb{Q}}(N)$ is not necessarily
closed under addition; if it happens to be so, $N$ is said to be a
\emph{lattice group}, and then $\log N$ is indeed a full lattice in
$\mathcal{L}_{\mathbb{Q}}(N)$.

An automorphism $\alpha$ of $N$ induces a Lie algebra automorphism, also
denoted $\alpha$, of $\mathcal{L}_{\mathbb{Q}}(N)$. This has a multiplicative
Jordan decomposition%
\[
\alpha=\alpha_{u}\cdot\alpha_{s}=\alpha_{s}\cdot\alpha_{u}%
\]
where $\alpha_{u}$ is unipotent and $\alpha_{s}$ is semisimple (see e.g.
\cite{S1} Chapter 7). Then $\alpha_{u}$ and $\alpha_{s}$ correspond to
automorphisms (with the same names) of $N^{\mathbb{Q}}$; these may or may not
fix $N.$ The automorphism $\alpha$ of $N$ is said to be semisimple if
$\alpha=\alpha_{s}$.

\begin{lemma}
\label{Jordanfor beta}\bigskip Let $\alpha$ be an automorphism of $N$. Then%
\[
(\alpha\beta_{N})_{u}=\alpha_{u}\beta_{N},~(\alpha\beta_{N})_{s}=\alpha
_{s}\beta_{N}.
\]

\end{lemma}

\begin{proof}
Write $\beta=\beta_{N}$ and $L=\mathbb{Q}\log(N\beta)\subseteq\mathfrak{u}%
_{n}(\mathbb{Q})$. Let $\widetilde{\Gamma}$ be as in Lemma \ref{faithful};
this is contained in the stabilizer of $L$ in the conjugation action of
$\mathrm{GL}_{n}(\mathbb{Q})$. Now $\alpha\beta\in\widetilde{\Gamma}$ so also
$(\alpha\beta)_{s},~(\alpha\beta)_{u}\in\widetilde{\Gamma}$ (the Jordan
components of $\alpha\beta$ as an element of $\mathrm{GL}_{n}(\mathbb{Q})$ ).
As the conjugation action on $L$ is rational, it follows that $(\alpha
\beta)_{s}$ and $(\alpha\beta)_{u}$ act on $L$ as they should, i.e.
semisimply, unipotently respectively (\cite{W} Chap. 7, and Theorem 14.20). On
the other hand, $\alpha_{s}\beta$ and $\alpha_{u}\beta$ act semisimply,
unipotently respectively on $L$ by definition. The four matrices all commute
and%
\[
(\alpha\beta)_{s}(\alpha\beta)_{u}=\alpha\beta=(\alpha_{s}\beta)(\alpha
_{u}\beta),
\]
so $(\alpha_{s}\beta)^{-1}(\alpha\beta)_{s}=(\alpha_{u}\beta)(\alpha\beta
)_{u}^{-1}$ acts both semisimply and unipotently, hence trivially, on $L$. The
result now follows by Lemma \ref{faithful}.
\end{proof}

\bigskip

The same considerations apply with $\mathbb{Q}_{p}$ in place of $\mathbb{Q}$.
For a lattice group $N$, the pro-$p$ completion $\widehat{N}_{p}$ may be
identified with $\exp(\mathbb{Z}_{p}\log N)\subseteq\exp\mathcal{L}%
_{\mathbb{Q}_{p}}(N)$. The embedding $\beta$ extends to an isomorphism from
$\widehat{N}_{p}$ to the closure of $N\beta$ in $\mathrm{U}_{n}(\mathbb{Z}%
_{p})$, and then to an isomorphism
\[
\beta:\widehat{N}=\prod\limits_{p}\widehat{N}_{p}\overset{\cong}{\rightarrow
}\overline{N\beta}\leq\mathrm{U}_{n}(\widehat{\mathbb{Z}})<\mathrm{GL}%
_{n}^{\infty}\text{.}%
\]

An automorphism $\alpha$ of $\widehat{N}$ fixes each $\widehat{N}_{p}$, and
has a multiplicative Jordan decomposition defined by its action on the
indvidual factors. Moreover, the action of $\alpha$ on $\overline{N\beta}$ is
induced by conjugation with an element of $\mathrm{GL}_{n}^{\infty}$.

\begin{lemma}
\label{Jordaninf}\emph{(\cite{S1} Chapter 10, Proposition 8)} If
$X<\mathrm{GL}_{n}(\mathbb{Z})$ is nilpotent, $X_{s}\subseteq\mathrm{GL}%
_{n}(\mathbb{Z})$ and $X^{\prime}$ is unipotent then%
\[
\overline{X}_{s}=\overline{X_{s}},~~\overline{X}_{u}=\overline{X_{u}%
}<\mathrm{GL}_{n}^{\infty}\text{.}%
\]

\end{lemma}

More generally, if $N^{\dag}$ is another $\mathcal{T}$-group and
$\theta:\widehat{N}\rightarrow\widehat{N^{\dag}}$ is an isomorphism, then
$n(N)=n(N^{\dag})$ and there exists $y\in\mathrm{GL}_{n}^{\infty}$ making the
following diagram commute:%
\[%
\begin{array}
[c]{ccc}%
\widehat{N} & \overset{\beta_{N}}{\longrightarrow} & \overline{N\beta_{N}}\\
\theta\downarrow &  & \downarrow y^{\ast}\\
\widehat{N^{\dag}} & \overset{\beta_{N^{\dag}}}{\longrightarrow} &
\overline{N^{\dag}\beta_{N^{\dag}}}%
\end{array}
;
\]
here $y$ represents the isomorphism $V(N)\otimes\widehat{\mathbb{Z}%
}\rightarrow V(N^{\dag})\otimes\widehat{\mathbb{Z}}$ induced by $\theta$ (see
\cite{S1},\ Chapter 10, Lemma 1).

If $N$ is a normal subgroup of a group $G,$ we can embed $G$ canonically into
a bigger group $N^{\mathbb{Q}}G,$ so that $N^{\mathbb{Q}}\cap G=N$ and the
following diagram of inclusions commutes:%
\[%
\begin{array}
[c]{ccc}%
N & \longrightarrow & N^{\mathbb{Q}}\\
\downarrow &  & \downarrow\\
G & \longrightarrow & N^{\mathbb{Q}}G
\end{array}
.
\]
In what follows we consider $G$ as a subgroup of $N^{\mathbb{Q}}G.$ For each
$m\in\mathbb{N},$ $G$ then normalizes the subgroup $N^{1/m}$ of $N^{\mathbb{Q}%
},$ and $\left\vert N^{1/m}G:G\right\vert =\left\vert N^{1/m}:N\right\vert
<\infty$ (\emph{cf}. \cite{GPS}, Lemmas 2.10, 2.7).

Mal'cev completion is a right-exact functor on $\mathcal{T}$-groups;
in particular if $M\vartriangleleft N$ and $N/M$ is torsion free there is a
natural surjection $N^{\mathbb{Q}}\rightarrow(N/M)^{\mathbb{Q}}$ (see
\cite{Wa} Theorem 8.11, Lemma 8.10, or \cite{Ba}\ Corollary 2.41).

The following is often useful:

\begin{lemma}
\label{nilpaction}\emph{(\cite{S1} Chapter 3, Lemma 1; \cite{LR} 10.3.7)}
\ Let $G$ be a nilpotent group and $A\cong\mathbb{Z}^{n}$ a $G$-module. Then
$\mathrm{C}_{G}(A)=H^{0}(G,A)=0$ if and only if $A/[A,G]=H_{0}(G,A)$ is
finite, and these imply that $\mathrm{C}_{G}(Q)=0$ for each $\mathbb{Z}%
$-torsion-free quotient module ~$Q$ of $A$.
\end{lemma}

\section{Semisimple splitting\label{ssspl}}

\subsection{Nilpotent supplements}

For convenience, let us say that that a virtually polycyclic group $G$ is
\emph{nice} if $\mathrm{Fit}(G)$ is a $\mathcal{T}$-group and $G/\mathrm{Fit}%
G)$ is free abelian. Basic results (\cite{S1} Chapter 1) imply that every
virtually polycyclic group $G_{1}$ has a nice characteristic subgroup $G$ of
finite index. It follows from \cite{S1} Chapter 3, Theorem 3, that $G$ has a
nilpotent subgroup $C$ such that $\mathrm{Fit}(G)C$ has finite index in $G$.
Such a $C$ is called a \emph{nilpotent supplement} to $\mathrm{Fit}(G)$ in
$\mathrm{Fit}(G)C$.

For any group $G$ and~$N\vartriangleleft G$, we denote by
\[
\mathcal{X}(G,N)
\]
the set of \emph{maximal nilpotent supplements} to $N$ in $G$. This set is
obviously invariant under conjugation by $G,$ and if $C\in\mathcal{X}(G,N)$
and $g=cv$ ($v\in N,~c\in C$) then $C^{g}=C^{v}$, so the $G$-orbits are the
same as the $N$-orbits. The set of orbits will be denoted%
\[
\mathcal{X}(G,N)/G=\mathcal{X}(G,N)/N.
\]
Finally, if $G$ is profinite it is clear that $\mathcal{X}(G,N)$ consists of
\emph{closed} subgroups. I shall use these observations without special mention.

\begin{lemma}
\label{maxymaxy}Let $G=NC$ where $G$ is polycyclic, $N\vartriangleleft G$ and
$N$, $C$ are nilpotent. Then%
\[
C\in\mathcal{X}(G,N)\Longleftrightarrow C=\mathrm{N}_{G}(C)\Longleftrightarrow
\overline{C}=\mathrm{N}_{\widehat{G}}(\overline{C})\Longleftrightarrow
\overline{C}\in\mathcal{X}(\widehat{G},\overline{N}).
\]

\end{lemma}

\begin{proof}
Put $D=\mathrm{N}_{N}(C)$. Then $DC$ is nilpotent, so $C$ is maximal nilpotent
if and only if $D\leq C.$ This implies the first equivalence since
$\mathrm{N}_{G}(C)=DC$. The last equivalence follows similarly from the fact
that $\widehat{G}=\overline{N}\overline{C}$. The middle one follows by
(\ref{subgps}).
\end{proof}

\bigskip

The next result, together with its proof, is essentially a translation into
the profinite realm of \cite{S1} Chap. 3, Thm. 4.

\begin{proposition}
\label{compconj}Let $G$ be a polycyclic group, $N$ a normal $\mathcal{T}%
$-subgroup of $G$, and $C\in\mathcal{X}(G,N)$. Then there is a finite subset
$X$ of $N^{\mathbb{Q}}$ such that
\begin{align}
\mathcal{X}(G,N)  &  =\bigcup\limits_{x\in X}\left\{  C^{xy}\mid y\in
N\right\}  ,\label{Xcong}\\
\mathcal{X}(\widehat{G},\overline{N})  &  =\bigcup\limits_{x\in X}\left\{
\overline{C}^{xy}\mid y\in\overline{N}\right\}  . \label{Xhatconj}%
\end{align}

\end{proposition}

\begin{proof}
I give the argument for (\ref{Xhatconj}); (\ref{Xcong}) will follow in the
same way, just ignoring `bars' and `hats'.

Consider first the special case where $N$ is abelian, and rationally
irreducible and non-trivial as a $C$-module. This implies that $N\cap
\mathbf{Z}(C)=N\cap\mathbf{Z}(G)=1$, which implies that $N\cap C=1$. So
$G=N.C$ is a semidirect product. Put $K=\mathrm{C}_{C}(N)$ and choose
\thinspace$b\in C\smallsetminus K$ so that $[b,C]\subseteq K$.

Write $N$ additively for the moment. Then $N(b-1)\geq mN$ for some
$m\in\mathbb{N}$. Let $X_{0}$ be a transversal to the cosets of $mN$ in $N$;
then $X_{1}:=m^{-1}X_{0}$ is a transversal to the cosets of $N$ in $m^{-1}N$,
and also to the cosets of $\overline{N}$ in $m^{-1}\overline{N}=\overline
{N}+m^{-1}N$. Let
\[
X=\{x\in X_{1}\mid C^{x}\in\mathcal{X}(G,N)\}.
\]

Suppose that $\widehat{G}=\overline{N}E$, where $E$ is closed and nilpotent.
Then $\mathrm{Z}(E)\cap\overline{N}\leq\mathrm{Z}(\widehat{G})\cap\overline
{N}=1$, which implies that $E\cap\overline{N}=1$. So $\widehat{G}=\overline
{N}.E=\overline{N}.\overline{C}$ is a semidirect product in two ways. Define
the $1$-cocycle $\delta:\overline{C}\rightarrow\overline{N}$ by
\[
a.a\delta\in E~~(a\in\overline{C}).
\]
Note that then
\begin{equation}
E=\{a.a\delta\mid a\in\overline{C}\}. \label{wholeE}%
\end{equation}

Now $\overline{K}\cap\ker\delta\vartriangleleft\overline{C}$, so if
$\overline{K}\delta\neq0$ there exists $a\in\overline{K}$ with $a\delta\neq0$
and $[a,\overline{C}]\subseteq\ker\delta$. Then for $c\in\overline{C}$ we have
$(ac)\delta=(ca)\delta$ which gives%
\[
a\delta c=(ac)\delta-c\delta=(ca)\delta-c\delta=a\delta
\]
whence $a\delta\in C_{\overline{N}}(\overline{C})=0,$ contradiction. It
follows that $\overline{K}\delta=0,$ and as $[b,\overline{C}]\subseteq
\overline{K}$ we have $(ab)\delta=(ba)\delta$ for all $a\in\overline{C}$.

Now $mb\delta\in m\overline{N}\leq\overline{N}(b-1)$ so $b\delta=(v-x)(b-1)$
for some $v\in\overline{N}$ and $x\in X_{1}.$

Let $a\in\overline{C}$. Then $(ab)\delta=(ba)\delta$ implies%
\[
a\delta(b-1)=b\delta(a-1)=(v-x)(b-1)(a-1).
\]
As $[a,b]\in\overline{K}$ and $\mathrm{C}_{\overline{N}}(b)=0$ it follows that
$a\delta=(v-x)(a-1).$ Reverting to multiplicative notation we see that%
\[
a.a\delta=a.(vx^{-1})^{a}(vx^{-1})^{-1}=va^{x}v^{-1};
\]
with (\ref{wholeE}) this implies that $E^{v}=\overline{C}^{x}$.

It remains to show that $x\in X$. Put $G_{1}=N^{1/m}G$, so $X\subset G_{1}$.
We have%
\[
C^{x}\leq G_{1}\cap(\overline{N}E^{v})\leq G_{1}\cap\overline{G}=G,
\]
so $C^{x}\in\mathcal{X}(G,N)$ by Lemma \ref{maxymaxy}.

For the general case we argue by induction on $h(N)$. Of course if $N=1$ the
result is trivial: $G$ (resp. $\overline{G}$) is the unique maximal nilpotent
subgroup of $G$ (resp. $\widehat{G}$), and we take $X=\{1\}$.

Suppose first that $Q:=N\cap\mathrm{Z}(G)\neq1$. Note that $N/Q$ is again
torsion-free. We apply the inductive hypothesis to $G/Q,$ noting that $C\geq
Q\,,$ $Q$ is centralized by $N^{\mathbb{Q}},$ and $\overline{Q}$ is contained
in every maximal closed nilpotent subgroup of $\widehat{G}$. A preimage in
$N^{\mathbb{Q}}$ of a suitable finite subset of $(N/Q)^{\mathbb{Q}}$ clearly
does what is required for the set $X$.

Assume henceforth that $N\cap\mathrm{Z}(G)=1<N$.

Let $T/N^{\prime}$ be the torsion subgroup of $N/N^{\prime}$ and let $Z/T$ be
biggest $G$-submodule of $N/T$ on which $G$ acts nilpotently. Then $Z<N$ by
\cite{S1} Chap. 3, Prop. 3, and $N/Z$ is torsion-free. Choose
$M\vartriangleleft G$ with $Z\leq M<N$ maximal subject to $A:=N/M$ being
torsion free. Then $A$ is rationally irreducible as a $C$-module, and it is
nontrivial by Lemma \ref{nilpaction}. Write $\symbol{126}:G\rightarrow G/M$
for the quotient map, and also for the induced maps $\widehat{G}%
\rightarrow\widehat{G}/\overline{M}$ and $N^{\mathbb{Q}}\twoheadrightarrow
A^{\mathbb{Q}}$. Then \thinspace$\widetilde{G}=A\widetilde{C}$ satisfies the
conditions for the `special case', with $A=\widetilde{N}$ in place of $N$ and
$\widetilde{C}$ for $C$. Let $X^{\dag}\subset A^{\mathbb{Q}}$ be the resulting
finite set, and pick a finite subset $X$ of $N^{\mathbb{Q}}$ so that $X^{\dag
}=\widetilde{X}$.

Now suppose again that $\widehat{G}=\overline{N}E$, where $E$ is closed and
nilpotent. Then there exist $x\in X$ and $v\in\overline{N}$ such that%
\begin{align*}
\widetilde{N}\widetilde{C^{x}}  &  =\widetilde{G}\\
\widetilde{E^{v}}  &  =\widetilde{\overline{C}^{x}}.
\end{align*}
Thus $NC^{x}=G$ and%
\[
\overline{M}E^{v}=\overline{M}\overline{C}^{x}=\overline{MC^{x}}.
\]
I claim that the original hypotheses are now satisfied with $MC^{x}$ for $G$,
$M$ for $N$ and $C^{x}$ for $C$. Accepting this for now, we identify
$\overline{MC^{x}}$ with $\widehat{MC^{x}}$, and infer by inductive hypothesis
that there is a finite subset $Y$ of $M^{\mathbb{Q}}$ such that (a)
$MC^{xy}=MC^{x}$ for each $y\in Y$, and (b) every member $E$ of $\mathcal{X}%
(\overline{MC^{x}},\overline{M})$ is conjugate to $\overline{C^{xy}}$ for some
$y\in Y$. In particular, there exist $y\in Y$ and $w\in\overline{M}$ such that
$E^{vw}=\overline{C^{xy}}$.

Evidently the set $XY\subseteq N^{\mathbb{Q}}$ fulfils our requirements.

It remains to verify that $C^{x}$ is maximal nilpotent in $MC^{x}$. To see
this, suppose that $C^{x}<D\leq G$ and $D$ is nilpotent. Then $D=(N\cap
D)C^{x}$ so $N\cap D\neq1$ and then $N\cap\mathrm{Z}(D)\neq1$. Hence
$1\neq\mathrm{Z}(N)\cap\mathrm{Z}(D)\leq N\cap\mathrm{Z}(G)$, contradicting hypothesis.
\end{proof}

\begin{corollary}
\label{iso-X}The groups in $\mathcal{X}(G,N)$ are all isomorphic.
\end{corollary}

\begin{corollary}
\label{finconjmaxns}With $G$ and $N$ as above, suppose that $\mathcal{X}(G,N)$
is nonempty. Then the mappings%
\begin{align*}
\mathfrak{u}  &  :C\longmapsto\overline{C}~~(C\in\mathcal{X}(G,N)~)\\
\mathfrak{d}  &  :E\longmapsto E\cap G~~(E\in\mathcal{X}(\widehat{G}%
,\overline{N})~)
\end{align*}
induce mutually inverse bijections $\mathcal{X}(G,N)/G\rightarrow
\mathcal{X}(\widehat{G},\overline{N})/\widehat{G}\rightarrow\mathcal{X}%
(G,N)/G$; these sets are finite.
\end{corollary}

\begin{proof}
It is obvious that $\mathfrak{u}$ respects conjugacy. Lemma \ref{maxymaxy}
shows that $\mathfrak{u}$ maps $\mathcal{X}(G,N)$ into $\mathcal{X}%
(\widehat{G},\overline{N})$. If $E\in\mathcal{X}(\widehat{G},\overline{N})$
then $E$ is conjugate to $\overline{C^{x}}$ for some $C\in\mathcal{X}(G,N)$
and $x\in N^{\mathbb{Q}}$, and then $C^{x}\in\mathcal{X}(G,N)$ by the same
corollary, so $E^{w}\in\mathcal{X}(G,N)\mathfrak{u}$ for some $w\in
\widehat{G}$. Thus $\mathfrak{u}$ is surjective on conjugacy classes, and it
is injective on conjugacy classes in view of (\ref{subconj}). From first
principles $\mathfrak{ud}$ is the identity on the set of all subgroups of $G,$
and the first claim follows. The second claim is immediate from Prop.
\ref{compconj}.
\end{proof}

\subsection{Splittings\label{sss&l}}

In \cite{GPS}, \S 7 we introduced a construction known as `semisimple
splitting' (a variation of an original construction due to L. Auslander). This
assigns to a (`splittable') polycyclic group a larger group of the form
$M\rtimes T,$ where $M$ is a $\mathcal{T}$-group and $T$ is an abelian
subgroup of $\mathrm{Aut}(M)$. This is explained again in\ \cite{S1} Chap. 7,
Section B, which may be consulted for details of some proofs.

In this subsection, we fix a nice polycyclic group $G$ with Fitting subgroup
$N$. Recall that $\mathcal{X}(G,N)$ denotes the set of maximal nilpotent
supplements to $N$ in $G$ (which may of course be empty in general). We set%
\[
\mathcal{X}^{\ast}(G,N)=\{C\in\mathcal{X}(G,N)\mid c_{s}^{\ast}~\text{fixes
}N~(\forall c\in C)\},
\]
where as usual $c^{\ast}$ denotes the action of $c$ by conjugation, and
$c_{s}^{\ast}$ its semisimple component as an automorphism of $N^{\mathbb{Q}}%
$. As $G=NC$ for each $C\in$ $\mathcal{X}(G,N)$ and $\mathcal{X}^{\ast}(G,N)$
is clearly invariant under $N^{\ast}$, we see that $\mathcal{X}^{\ast}(G,N)$
is invariant under conjugation by $G$.

The group $G$ is said to be \emph{splittable }if $G$ is nice and
$\mathcal{X}^{\ast}(G,N)$ is nonempty; $G$ is \emph{lattice-splittable }if, in
addition, $N=\mathrm{Fit}(G)$ is a lattice $\mathcal{T}$-group (\emph{cf}.
\cite{S1} Chapter 10, \S D).

The set $\mathcal{X}^{\ast}(\widehat{G},\overline{N})$ is defined analogously:
for each prime $p$ the $p$-component $(c^{\ast})_{p,s}$ is supposed to fix
$\mathbb{Z}_{p}\log N$ inside $\mathbb{Q}_{p}\log N$.

\begin{lemma}
\label{X*conj}Assume that $N$ is a lattice group. Provided $\mathcal{X}%
(G,N)\neq\varnothing$, the map $D\longmapsto\overline{D}$ induces a bijection
$\mathcal{X}^{\ast}(G,N)/G\rightarrow\mathcal{X}^{\ast}(\widehat{G}%
,\overline{N})/\widehat{G}$.
\end{lemma}

\begin{proof}
Lemma \ref{Jordaninf} shows that for $D\in\mathcal{X}(G,N)$, we have
$D\in\mathcal{X}^{\ast}(G,N)$ if and only if $\overline{D}\in\mathcal{X}%
^{\ast}(\widehat{G},\overline{N})$. The claim now follows from Corollary
\ref{finconjmaxns}.
\end{proof}

\bigskip\bigskip

We assume for the rest of this subsection that $G$ is
\emph{lattice-splittable}$.$

Fix $C\in\mathcal{X}^{\ast}(G,N)$. For $c\in C$ we extend the automorphism
$c_{s}^{\ast}$ of $N$ to an automorphism $\tau(c)$ of $G$ that centralizes
$C$, and define $u(c)\in\mathrm{Hol}(G)$ as follows:%
\[
u(c)=c.\tau(c)^{-1}.
\]

Then $u:C\rightarrow\mathrm{Hol}(G)$ and $\tau:C\rightarrow\mathrm{Aut}(G)$
are homomorphisms.%
\[
M:=M_{C}:=N.u(C)\leq\mathrm{Hol}(G)
\]
is a $\mathcal{T}$-group and%
\[
T:=T_{C}:=\tau(C)\leq\mathrm{Aut}(G)
\]
is a free abelian group, acting trivially on both $C$ and $u(C)$, and acting
as $C_{s}^{\ast}$ on $N$. Note that $\ker\tau=C\cap N$ (because $N$ is the
Fitting subgroup of $G$), so $T\cong C/(C\cap N)\cong G/N$.

As $T$ centralizes $u(C)$ it acts faithfully as a group of semisimple
automorphisms on $M$, and we consider $M.T\cong M\rtimes T$ as a subgroup of
$\mathrm{Hol}(M)$. Clearly $M=\mathrm{Fit}(MT)$.

We have%

\begin{equation}
G=NC\leq Nu(C)\tau(C)=MT\leq\mathrm{Hol}(M), \label{Gin SSS}%
\end{equation}
and
\begin{equation}
MG=MC=M\rtimes T=G\rtimes T. \label{M.TvsG.T}%
\end{equation}
This group is the \emph{semisimple splitting} of $G$ associated to the chosen
subgroup $C\in\mathcal{X}^{\ast}(G,N)$.

Recall the embedding%
\begin{align*}
\beta_{M}:\mathrm{Hol}(M)  &  \rightarrow\mathrm{GL}_{n}(\mathbb{Z})\\
M\beta_{M}  &  \leq\mathrm{U}_{n}(\mathbb{Z})
\end{align*}
where $n=n(M).$ Restricting $\beta_{M}$ to $MG$ gives an embedding%
\[
\beta_{G,C}:MG\rightarrow\mathrm{GL}_{n}(\mathbb{Z}),
\]
$n=n(M_{C})$. In fact we have

\begin{lemma}
$n(M_{C})$ is independent of $C\in\mathcal{X}^{\ast}(G,N)$
\end{lemma}

\begin{proof}
By Proposition \ref{compconj} it suffices to check that $n(M_{C})=n(M_{C^{y}%
})$ if $y\in N^{\mathbb{Q}}$ and $C^{y}\in\mathcal{X}^{\ast}(G,N)$. Say $y\in
N^{1/m}$ and put $G_{1}=N^{1/m}G.~$Then $M_{C}\leq_{f}M_{1}:=\mathrm{Fit}%
(G_{1}T_{C})$ and $M_{C^{y}}\leq_{f}M_{2}:=\mathrm{Fit}(G_{1}T_{C^{y}})$,
while
\[
G_{1}T_{C^{y}}=(G_{1}T_{C})^{y^{\ast}}\leq G_{1}\rtimes\mathrm{Aut}(G_{1}).
\]
So $M_{1}\cong M_{2}$, whence $M_{C^{y}}$ and $M_{C}$ are abstractly
commensurable groups. The claim now follows from \cite{S1} Chapter 5, Theorem 3.
\end{proof}

\bigskip

From now on, we accordingly write $n(G)$ for $n(M_{C})$, where $C\in
\mathcal{X}^{\ast}(G,N)$ can be arbitrary (this is consistent since if
$G\in\mathcal{T}$ then $M=G$).

\begin{lemma}
\label{Tdiag}The group $T\beta_{M}$ is diagonalizable.
\end{lemma}

\begin{proof}
Lemma \ref{Jordanfor beta} implies that $T\beta$ consists of semisimple
elements, and the result follows since $T\beta$ is abelian.
\end{proof}

\begin{lemma}
\label{diag}Let $T,~T^{\dag}$ be diagonalizable subgroups of $\mathrm{GL}%
_{n}(\mathbb{Z})$. If $\overline{T}^{x}=\overline{T^{\dag}}$ with
$x\in\mathrm{GL}_{n}^{\infty}$ then%
\begin{equation}
T^{x}=T^{\dag}. \label{T-conj}%
\end{equation}

\end{lemma}

\begin{proof}
There exist $\alpha,~\beta\in\mathrm{GL}_{n}(k)$, $k$ some algebraic number
field, such that $T^{\alpha},~T^{\beta}\leq\mathrm{D}_{n}(\mathfrak{o}_{k})$.
The proof of \cite{GS2}, Lemma 5.4 shows that for some $\sigma\in
\mathrm{Sym}(n),$ the matrix%
\[
\alpha^{-1}x\beta\sigma^{-1}%
\]
centralizes $T^{\alpha}$ (this matrix sits in the finite adele group over $k$
of $\mathrm{GL}_{n}$, which contains $\mathrm{GL}_{n}^{\infty}$ as well as
$\alpha$ and $\beta$). Then for $t\in T$ we have%
\[
t^{x\beta}=t^{\alpha\sigma}\in\mathrm{D}_{n}(\mathfrak{o}_{k})\cap
\overline{T^{\dag\beta}}=T^{\dag\beta}%
\]
because $T^{\dag\beta}$ is closed in $\mathrm{D}_{n}(\mathfrak{o}_{k})$: in
the profinite topology because $\mathrm{D}_{n}(\mathfrak{o}_{k})$ is a
finitely generated abelian group, and in the congruence topology because this
induces the former. Thus $T^{x}\leq T^{\dag}$ and (\ref{T-conj}) follows by symmetry.

(The first step of this argument depends on some serious number theory, the
main result of \cite{GS4}.)
\end{proof}

\bigskip

\emph{Remark} Almost the same conclusion may be deduced from an apparently
weaker hypothesis $\overline{MG}^{y}=\overline{M^{\dag}G^{\dag}}$; using the
fact that for almost all primes $\mathfrak{p}$ of $\mathfrak{o}_{k},$ the
images $T^{\alpha}\pi_{\mathfrak{p}}$ and $T^{\dag\beta}\pi_{\mathfrak{p}}$
are Hall $p$-complements in the finite groups $(MG)^{\alpha}\pi_{\mathfrak{p}%
}$, $(M^{\dag}G^{\dag})^{\beta}\pi_{\mathfrak{p}}$, one finds $x\equiv
y~(\operatorname{mod}$ $\overline{M})$ that satisfies (\ref{T-conj}). Exercise
for the reader! (\emph{cf}. \cite{S1} Chapter 9, Proposition 6).

\bigskip

From now on, suppose that $G^{\dag}$ is another lattice-splittable group, and
$N^{\dag}=\mathrm{Fit}(G^{\dag})$.

\begin{proposition}
\label{iso-conj} Let $C\in\mathcal{X}^{\ast}(G,N),~C^{\dag}\in\mathcal{X}%
^{\ast}(G^{\dag},N^{\dag})$ and set $M=M_{C}\leq\mathrm{Hol}(G)$, $M^{\dag
}=M_{C^{\dag}}\leq\mathrm{Hol}(G^{\dag})$.

\begin{description}
\item[\textbf{(i)}] Suppose that $n(G)=n(G^{\dag})=n$. Then $\beta_{G,C}$ and
$\beta_{G^{\dag},C^{\dag}}$ extend to embeddings%
\[
\beta=\beta_{G,C}:\widehat{MG}\rightarrow\mathrm{GL}_{n}^{\infty},~\beta
^{\dag}=\beta_{G^{\dag},C^{\dag}}:\widehat{M^{\dag}G^{\dag}}\rightarrow
\mathrm{GL}_{n}^{\infty}~.
\]

\item[\textbf{(ii)}] Suppose that $\theta:\widehat{G}\rightarrow
\widehat{G^{\dag}}$ is an isomorphism such that $\overline{C}\theta
=\overline{C^{\dag}}$. Then $n(G)=n(G^{\dag})=n$, say, $\theta$ extends to an
isomorphism $\widehat{MG}\rightarrow\widehat{M^{\dag}G^{\dag}}$ such that
\begin{equation}
\overline{M}\theta=\overline{M^{\dag}},~\ \overline{T_{C}}\theta
=\overline{T_{C^{\dag}}}, \label{goodimages}%
\end{equation}
and there exists $x\in\mathrm{GL}_{n}^{\infty}$ such that the following
diagram commutes:%
\begin{equation}%
\begin{array}
[c]{ccc}%
\widehat{MG} & \overset{\beta}{\longrightarrow} & \mathrm{GL}_{n}^{\infty}\\
{\small \theta}\downarrow &  & \downarrow{\small x}^{\ast}\\
\widehat{M^{\dag}G^{\dag}} & \underset{\beta^{\dag}}{\longrightarrow} &
\mathrm{GL}_{n}^{\infty}%
\end{array}
. \label{theta-x*}%
\end{equation}
In particular,%
\begin{equation}
(\overline{G\beta})^{x}=\overline{G^{\dag}\beta^{\dag}}. \label{Gbar-cong}%
\end{equation}

\end{description}
\end{proposition}

\begin{proof}
(i) follows from general principles. For (ii): as $\theta$ maps $\overline
{N}=\mathrm{Fit}(\widehat{G})$ to $\overline{N^{\dag}}=\mathrm{Fit}%
(\widehat{G^{\dag}})$ and $\overline{C}$ to $\overline{C^{\dag}},$ it will map
the image of $\overline{C}^{\ast}$ in $\mathrm{Aut}(\overline{N})$ to the
image of $\overline{C^{\dag}}^{\ast}$ in $\mathrm{Aut}(\overline{N^{\dag}}),$
and this mapping will respect the Jordan decomposition (\emph{cf}. \cite{S1}
Chapter 10, Proposition 8). We can then verify that for $c\in\overline{C},$%
\[
\tau(c)\theta=\tau(c\theta),~u(c)\theta=u(c\theta)~,
\]
where $\tau$ and $u$ are extended in the natural way to $\overline{C}.$ This
implies (\ref{goodimages}). The rest of (ii) now follows from the proof of
\cite{S1} Chapter 10, Lemma 7, or the argument in \cite{GPS} \S 14.
\end{proof}

\begin{lemma}
\label{C-choices}Let $\{D_{1},\ldots,D_{k}\},$ $\{D_{1}^{\dag},\ldots
,D_{k^{\prime}}^{\dag}\}$ be complete sets of representatives for
$\mathcal{X}^{\ast}(G,N)/G$, respectively $\mathcal{X}^{\ast}(G^{\dag}%
,N^{\dag})/G^{\dag}$. Suppose $\theta:\widehat{G}\rightarrow\widehat{G^{\dag}%
}$ is an isomorphism. Then $k^{\prime}=k,$ and there exists $\rho
\in\mathrm{Sym}(k)$ such that%
\begin{equation}
(\overline{D_{i}}^{x_{i}})\theta=\overline{D_{i\rho}^{\dag}}~~~(i=1,\ldots,k)
\label{Diperm}%
\end{equation}
for some $x_{1},\ldots,x_{k}\in\overline{N}$.
\end{lemma}

\begin{proof}
Follows from Lemma \ref{X*conj}, noting as before that if $\overline{D}%
\in\mathcal{X}(\widehat{G},\overline{N})$ and $g\in\widehat{G}$ then
$\overline{D}^{g}=\overline{D}^{x}$ where $x\in\overline{N}$ and
$g\in\overline{D}x$.
\end{proof}

\begin{corollary}
\label{TandTdag}\bigskip Suppose that $\theta:\widehat{G}\rightarrow
\widehat{G^{\dag}}$ is an isomorphism. Then%
\begin{equation}
(\overline{N}G)\theta=\overline{N^{\dag}}G^{\dag}. \label{NbarGetc}%
\end{equation}

\end{corollary}

\begin{proof}
By Proposition \ref{C-choices} there exist $C\in\mathcal{X}^{\ast}(G,N)$,
$C^{\dag}\in\mathcal{X}^{\ast}(G^{\dag},N^{\dag})$ and $y\in\overline{N}$ such
that $\overline{C}^{y}\theta=\overline{C^{\dag}}$. Write $M=M_{C},~T=T_{C}$
etc. Then $\theta_{1}:=y^{\ast}\circ\theta$ extends to an isomorphism
$\widehat{MG}\rightarrow\widehat{M^{\dag}G^{\dag}}$. Recall that $MG=MT$ and
$M\cap G=N$, whence $\overline{N}G=\overline{M}G\cap\overline{G}=\overline
{M}T\cap\overline{G}.$ Now (\ref{theta-x*}) and Lemma \ref{diag} give
\begin{align*}
(\overline{N}G)\theta_{1}=(\overline{M}T)\theta_{1}\cap\overline{G}\theta &
_{1}=\overline{M^{\dag}}T^{\dag}\cap\overline{G^{\dag}}\\
&  =\overline{M^{\dag}}G^{\dag}\cap\overline{G^{\dag}}=(\overline{M^{\dag}%
}\cap\overline{G^{\dag}})G^{\dag}=\overline{N^{\dag}}G^{\dag}.
\end{align*}
The result follows since $(\overline{N}G)^{y}=\overline{N}G$.
\end{proof}

\subsection{\bigskip Proposition \ref{finiteexrtrabit}\label{mainproppf}}

We have to generalize Corollary \ref{TandTdag}. Specifically, $G$
and$~G^{\dag}$ are virtually polycyclic groups, $G_{0}\leq G$ and $G_{0}%
^{\dag}\leq G^{\dag}$ are characterisitc lattice-splittable subgroups of
finite index, and $\theta:\widehat{G}\rightarrow\widehat{G^{\dag}}$ is an
isomorphism that maps $\widehat{G_{0}}$ onto $\widehat{G_{0}^{\dag}}$. Also
$N=\mathrm{Fit}(G_{0}),~N^{\dag}=\mathrm{Fit}(G_{0}^{\dag})$.

\bigskip The aim is to prove that there exists $v\in\overline{G_{0}}$ \ such
that $\Theta:=v^{\ast}\circ\theta$ satisfies%
\begin{equation}
(\overline{N}G)\Theta=\overline{N^{\dag}}G^{\dag}. \label{finalgoodthing}%
\end{equation}

\begin{proof}
Proposition \ref{TandTdag} shows that $(\overline{N}G_{0})\theta
=\overline{N^{\dag}}G_{0}^{\dag}$. Put $L=(\overline{N^{\dag}}G^{\dag}%
)\theta^{-1}$. Then
\[
(L\cap\overline{G_{0}})\theta=\overline{N^{\dag}}G^{\dag}\cap\overline
{G_{0}^{\dag}}=\overline{N^{\dag}}(G^{\dag}\cap\overline{G_{0}^{\dag}%
})=\overline{N^{\dag}}G_{0}^{\dag}=(\overline{N}G_{0})\theta
\]
and%
\[
\widehat{G}\theta=\widehat{G^{\dag}}=\overline{G_{0}^{\dag}}G^{\dag}%
=\overline{G_{0}^{\dag}}.\overline{N^{\dag}}G^{\dag}=\overline{G_{0}}%
\theta.L\theta.
\]
So%
\[
\overline{G_{0}}L=\widehat{G},~~\overline{G_{0}}\cap L=\overline{N}G_{0}.
\]
As both $G_{0}$ and $N$ are normal in $G$ and $\overline{N}G_{0}%
\vartriangleleft\overline{G_{0}}$ (since $\overline{G_{0}}/\overline{N}$ is
abelian), we see that $\overline{N}G_{0}\vartriangleleft\overline{G}$. Thus
both $\overline{N}G/\overline{N}G_{0}$ and $L/\overline{N}G_{0}$ are
complements to $\overline{G_{0}}/\overline{N}G_{0}$ in $\overline{G}%
/\overline{N}G_{0}$. Now $\overline{G_{0}}/\overline{N}G_{0}\cong%
(\widehat{\mathbb{Z}}/\mathbb{Z})^{d}$ is a divisible abelian group (where
$G_{0}/N\cong\mathbb{Z}^{d}$); consequently $H^{1}(\overline{G}/\overline
{G_{0}},\overline{G_{0}}/\overline{N}G_{0})=0$ as $\overline{G}/\overline
{G_{0}}\cong G/G_{0}$ is finite (\cite{B}, Corollary 10.2). This implies that
the two complements just indicated are conjugate, which translates directly to
the claim.
\end{proof}

\subsection{Finding lattice-splittable subgroups 1:
effectiveness\label{finding}}

Here we establish Proposition \ref{first step}. Here, $G$ and $G^{\dag}$ are
virtually polycyclic groups, each given by a finite presentation. Assuming
that $\widehat{G^{\dag}}\cong\widehat{G},$ we have to construct
lattice-splittable characteristic subgroups $G_{0}\leq G,~G_{0}^{\dag}\leq
G^{\dag}$ of finite index such that every isomorphism $\widehat{G}%
\rightarrow\widehat{G^{\dag}}$ sends $\overline{G_{0}}$ onto $\overline
{G_{0}^{\dag}}$.

This is a small tweak to \cite{S3}, Proposition 4.6.

For convenience, let us call a \emph{subgroup functor} a mapping that assigns
to each virtually polycyclic group a specific subgroup. A subgroup functor $F$
will be called $\mathcal{P}$\emph{-invariant} if whenever $\theta
:\widehat{G}\rightarrow\widehat{G^{\dag}}$ is an isomorphism one has
$\overline{F(G)}\theta=\overline{F(G^{\dag})}$. Thus (\ref{Fit}) and
(\ref{powers}) imply that $G\longmapsto\mathrm{Fit}(G)$ and $G\longmapsto
G^{m}$ are $\mathcal{P}$-invariant subgroup functors. Similarly, a property or
numerical invariant is $\mathcal{P}$\emph{-invariant }if it is preserved on
passing from $G$ to another group $G^{\dag}$ with $\widehat{G^{\dag}}%
\cong\widehat{G}$.

\bigskip

\emph{Step 1. }Set $F=\mathrm{Fit}(G)$, $F^{\dag}=\mathrm{Fit}(G^{\dag})$. Let
$e$ be the least positive integer such that $N_{1}:=F^{e}$ is torsion-free.
Determine $n(N_{1})$ (see \cite{GS1}, \S 8.3). According to \cite{S1} Chapter
6, Theorem 5, there is a number $\alpha,$ depending only on $n(N_{1}),$ such
that $N:=N_{1}^{\alpha}$ is a lattice group. Let $f$ be the least positive
integer such that $NG^{f}/N$ is free abelian. Set
\[
G_{1}=NG^{f}.
\]

All of $e,~f,~N,~G_{1}$ can be found effectively, using results from
\cite{BCRS}. Note that $\mathrm{Fit}(G_{1})=F\cap G_{1}=N$ since $G_{1}/N$ is
free abelian while $F/N$ has finite exponent.

A straightforward application of the `Profinite technicalities' shows that
$\mathrm{Fit}$ and $e$ are $\mathcal{P}$-invariant, and $n(N_{1})$ is
$\mathcal{P}$-invariant by \cite{S1}, Chapter 10, Lemma 1. So $N$ and $f$ are
also $\mathcal{P}$-invariants of $G$, and $G\longmapsto G_{1}$ is
$\mathcal{P}$-invariant. This implies that if we define $G_{1}^{\dag}%
:=N^{\dag}G^{\dag f}$ where $N^{\dag}=(F^{\dag})^{\alpha},$ then
$\mathrm{Fit}(G_{1}^{\dag})=N^{\dag}$ and $\overline{G_{1}}\theta
=\overline{G_{1}^{\dag}}$ for any isomorphism $\theta:\widehat{G}%
\rightarrow\widehat{G^{\dag}}$.

\bigskip

\emph{Step 2. }Now $G_{1}$ is nice and $N=\mathrm{Fit}(G_{1})$ is a lattice
group. A finite presentation for $G_{1}$ can be effectively constructed (See
\cite{BCRS} \S \S 2, 3). The analogous statements hold for $G_{1}^{\dag}$,
$N^{\dag}$. The (simple) proof of \cite{S3}, Proposition 4.6 shows that we can
effectively find $q,~q^{\prime}\in\mathbb{N}$ and nilpotent subgroups $C\leq
G_{1},~C^{\dag}\leq G_{1}^{\dag}$ such that
\[
G_{1}^{q}\leq NC,~G_{1}^{\dag q^{\prime}}\leq N^{\dag}C^{\dag}%
\]
and $(C^{\ast})_{s}$ fixes $N$, $(C^{\dag\ast})_{s}$ fixes $N^{\dag}$. Put
$\widetilde{q}=\operatorname{lcm}\{q,~q^{\prime}\}$. Then $G_{0}%
:=NG_{1}^{\widetilde{q}}$ and $G_{0}^{\dag}:=N^{\dag}G_{1}^{\dag\widetilde{q}%
}$ are lattice splittable, clearly characteristic of finite index in $G,$
$G^{\dag}$ respectively, and $\overline{G_{0}}\theta=\overline{G_{0}^{\dag}}$
for any $\theta$ as above.

This completes the proof.

\bigskip

\emph{Remark.} This assignment $G\longmapsto G_{0}$ is not \emph{a priori}
$\mathcal{P}$-invariant, because the choice of $\widetilde{q}$ depends on the
particular group $G^{\dag}$. To obtain a $\mathcal{P}$-invariant version needs
more work, see \S \ref{lastbit}.

\section{Going Diophantine\label{dio}}

It remains to establish Proposition \ref{equns}.

\bigskip

\emph{Data}

\begin{itemize}
\item Virtually polycyclic groups
\[
G^{\dag},~~G=\left\langle g_{1},\ldots,g_{d}~;~R\right\rangle
\]
where $R$ is a~finite set of words.

\item subgroups%
\[
N^{\dag}=\left\langle a_{1}^{\prime},\ldots,a_{s^{\prime}}^{\prime
}\right\rangle \vartriangleleft G^{\dag},~~N=\left\langle a_{1},\ldots
,a_{s}\right\rangle \vartriangleleft G.
\]

\item an isomorphism $\theta:G/N\rightarrow G^{\dag}/N^{\dag}.$
\end{itemize}

It is assumed that $N$ and $N^{\dag}$ are lattice $\mathcal{T}$-groups, and
that every isomorphism $\widehat{G}\rightarrow\widehat{G^{\dag}}$ sends
$\overline{N}$ onto $\overline{N^{\dag}}$.

We are given words $u_{i}$ such that
\[
a_{i}=u_{i}(g_{1},\ldots,g_{d})~\ \ (i=1,\ldots,s),
\]
and elements $h_{1},\ldots,h_{d}\in G^{\dag}$ such that%

\[
Ng_{i}\theta=N^{\dag}h_{i}~~~(i=1,\ldots,d).
\]
Define $b_{w},~c_{i}\in N^{\dag}$ by%
\begin{align*}
b_{w}  &  :=w(h_{1},\ldots,h_{d})~~(w\in R),\\
c_{i}  &  :=u_{i}(h_{1},\ldots,h_{d})~~(i=1,\ldots,s).
\end{align*}

The task is to construct a Diophantine system $\mathcal{F}$ that is locally
solvable if and only if the following holds: there exist an isomorphism
$\theta_{1}:\overline{N}\rightarrow\overline{N^{\dag}}$ and elements
$v_{1},\ldots,v_{d}\in\overline{N^{\dag}}$ that satisfy%
\begin{align}
w  &  \in R\Longrightarrow w_{\mathbf{h}}^{\prime}(v_{1},\ldots,v_{d}%
)b_{w}=1\tag{3}\label{(3)}\\
a_{i}\theta_{1}  &  =u_{i\mathbf{h}}^{\prime}(v_{1},\ldots,v_{d}%
).c_{i}~~~(i=1,\ldots,d). \tag{4}\label{(4)}%
\end{align}

\bigskip\emph{Introducing co-ordinates} \bigskip

Algorithm E of \cite{GS1} determines $n:=n(N)$ and $n^{\prime}:=n(N^{\dag})$,
and provides explicit monomorphisms $\beta=\beta_{N}:N\rightarrow
\mathrm{U}_{n}(\mathbb{Z})$ and $\beta^{\prime}=\beta_{N^{\dag}}:N^{\dag
}\rightarrow\mathrm{U}_{n^{\prime}}(\mathbb{Z}).$ These have the following
property (see \cite{GPS} Lemma 3.3 or \cite{S1} Chapter 10, Lemma 1, and
\cite{S1} Chapter 5, Theorem 2): if $\widehat{N}\cong\widehat{N^{\dag}}$ then
$n=n^{\prime}$ and for each isomorphism $\varphi:\widehat{N}\rightarrow
\widehat{N^{\dag}}$ there exists $y=y(\varphi)\in\mathrm{GL}_{n}^{\infty}$
such that%
\begin{equation}%
\begin{array}
[c]{ccc}%
\widehat{N} & \overset{\beta}{\longrightarrow} & \overline{N\beta}\\
{\small \varphi}\downarrow &  & \downarrow y^{\ast}\\
\widehat{N^{\dag}} & \overset{\beta^{\prime}}{\longrightarrow} &
\overline{N^{\dag}\beta^{\prime}}%
\end{array}
\label{hat-conj}%
\end{equation}
commutes, where the horizontal maps are the natural extensions of
$\beta,~\beta^{\prime}$. Conversely, if $n=n^{\prime}$ and $\overline{N^{\dag
}\beta^{\prime}}$ is conjugate to $\overline{N\beta}$ in $\mathrm{GL}%
_{n}^{\infty}$ then $\widehat{N}\cong\widehat{N^{\dag}}$.

Similarly, each automorphism $\alpha$ of $N^{\dag}$ corresponds via
$\beta^{\prime}$ to the conjugation action of $\mathbf{y}(\alpha)$ on
$N^{\dag}\beta^{\prime}$ for some $\mathbf{y}(\alpha)\in\mathrm{GL}%
_{n^{\prime}}(\mathbb{Z})$ (\cite{S1} Chapter 5, Theorem 2). If $\alpha$ is
given the algorithm finds $\mathbf{y}(\alpha)$ explicitly.

Since $N$ and $N^{\dag}$ are lattice groups the sets $L:=\log(N\beta)$ and
$L^{\prime}:=\log(N^{\dag}\beta^{\prime})$ are full lattices in the Lie
algebras $\mathcal{L}_{\mathbb{Q}}(N),~\mathcal{L}_{\mathbb{Q}}(N^{\dag})$.
They have explicitly known generating sets $\{\log(a_{i}\beta)\mid
i=1,\ldots,s\},~~\{\log(a_{i}^{\prime}\beta^{\prime})\mid i=1,\ldots
,s^{\prime}\}$, so by elementary algebra we may construct $\mathbb{Z}$-bases
$\{e_{1},\ldots,e_{r}\}$ for $L$ and $\{e_{1}^{\prime},\ldots,e_{r^{\prime}%
}^{\prime}\}$ for $L^{\prime}$ (thus $r,~r^{\prime}$ are the dimensions of
$\mathcal{L}_{\mathbb{Q}}(N),~\mathcal{L}_{\mathbb{Q}}(N^{\dag})$ ). Note that
$n!e_{i}\in\mathfrak{u}_{n}(\mathbb{Z}),~n^{\prime}!e_{i}^{\prime}%
\in\mathfrak{u}_{n^{\prime}}(\mathbb{Z})$.

If $r=r^{\prime}$, to a matrix $\mathbf{y}\in\mathrm{GL}_{n}^{\infty}$ that
conjugates $\overline{N\beta}$ to $\overline{N^{\dag}\beta^{\prime}}$ there
corresponds a matrix $\mathbf{z}=\mathbf{z}(\mathbf{y})\in\mathrm{GL}%
_{r}^{\infty}$ such that%
\begin{equation}
\mathbf{y}^{-1}e_{i}\mathbf{y}=\sum\limits_{j=1}^{r}z_{ij}e_{j}^{\prime
}~~~(i=1,\ldots,r). \label{linvsconj}%
\end{equation}

The unknown elements $v_{i}$ of $\overline{N^{\dag}}$ will be represented by
$n^{\prime}\times n^{\prime}$ matrices $\Xi(i)$, setting%
\[
\Xi(i):=v_{i}\beta^{\prime}.
\]
The unknown isomorphism $\theta_{1}:\overline{N}\rightarrow\overline{N^{\dag}%
}$ will be represented by an $n\times n^{\prime}$ matrix $H$; if $\theta_{1}$
exists then $n=n^{\prime}$ and we will set%
\[
H\mathbf{:}\mathbf{=y}(\theta_{1}).
\]

For any word $w$ in $d$ variables, $w_{\mathbf{h}}^{\prime}(v_{1},\ldots
,v_{d})$ is a finite product of terms of the form $v_{l(m)}^{w_{m}(h_{1}%
^{-1},\ldots,h_{d}^{-1})}$, where each $w_{m}$ is a subword of $w$.
Conjugation by $w_{m}(h_{1}^{-1},\ldots,h_{d}^{-1})$ induces an automorphism
$\alpha_{w,m}$ on $N^{\dag}$. If $w$ is given we may compute a matrix
\[
Y_{m}(w):=\mathbf{y}(\alpha_{w,m})\in\mathrm{GL}_{n^{\prime}}(\mathbb{Z})
\]
such that $Y_{m}(w)^{\ast}$ acts on $N^{\dag}\beta^{\prime}$ like $w_{m}%
(h_{1}^{-1},\ldots,h_{d}^{-1})^{\ast}$ on $N^{\dag}$. (In the following, $m$
will range from $1$ to the length of $w,$ so this varies from place to place).

Then for $v_{1},\ldots,v_{d}\in\overline{N^{\dag}}$ we have
\[
w_{\mathbf{h}}^{\prime}(v_{1},\ldots,v_{d})\beta^{\prime}=\prod\limits_{m}%
Y_{m}(w)^{-1}\cdot v_{l(m)}\beta^{\prime}\cdot Y_{m}(w).
\]

Thus (\ref{(3)}) is equivalent to%
\begin{equation}
\left(  \prod\limits_{m}Y_{m}(w)^{-1}\cdot\Xi(l(m))\cdot Y_{m}(w)\right)
\cdot B(w)=1~~~~(w\in R) \label{Beq}%
\end{equation}
\bigskip where $B(w)=b_{w}\beta^{\prime}\in\mathrm{GL}_{n^{\prime}}%
(\mathbb{Z})$. If $n=n^{\prime}$ and $r=r^{\prime}$ then (\ref{(4)}) is
equivalent to%
\begin{equation}
H\cdot\left(  \prod\limits_{m}Y_{m}(w)^{-1}\cdot\Xi(l(m))\cdot Y_{m}%
(w)\right)  \cdot C(i)=A(i)H~~~(i=1,\ldots,s). \label{Ceq}%
\end{equation}
where $C(i)=c_{i}\beta^{\prime}\in\mathrm{GL}_{n^{\prime}}(\mathbb{Z})$ and
$A(i)=a_{i}\beta\in\mathrm{GL}_{n}(\mathbb{Z})$.

\bigskip

\emph{The equations}

\bigskip

Each equation involving matrix unknowns should be interpreted as a family of
of $n^{2}$ equations, one for each component of the indicated matrices. Some
of the `constant' matrices that appear, and the polynomial $\log$, have
non-integral coefficients: to get equations with $\mathbb{Z}$-coefficients
just multiply each side by $n!$.

\bigskip The equations will only be interesting when $n=n^{\prime}$ and
$r=r^{\prime}$: I will call this \emph{the good case}.

\bigskip

\emph{The unknowns}: the entries of the matrices $H$ and $\Xi(1),\ldots
,\Xi(d)$; the entries $\zeta_{ij}$ of an $r\times r^{\prime}$ matrix $Z$,
which in the good case will stand for $\mathbf{z}(H)$ (as in (\ref{linvsconj}%
)); certain $\lambda(i)_{j}~(i=1,\ldots,d,~~j=1,\ldots,r^{\prime}),$ and
single unknowns $\eta_{0},\ \zeta_{0}$.

\bigskip

If we are not in the good case, let $\det H,~\det Z$ denote the determinants
of the biggest top-left minor (just to be specific).

\bigskip

Now $\mathcal{F}(\theta)$ is the following system of equations, in addition to
(\ref{Beq}) and (\ref{Ceq}):%
\begin{align}
n-n^{\prime}  &  =0,~~r-r^{\prime}=0\label{e1}\\
\eta_{0}\cdot\det H  &  =1,~~\zeta_{0}\cdot\det Z=1\label{e2}\\
e_{i}\cdot H  &  =H\cdot\sum_{j=1}^{r}\zeta_{ij}e_{j}^{\prime}%
~~~~\ (i=1,\ldots,r)\label{e3}\\
\log\Xi(i)  &  =\sum_{j=1}^{r^{\prime}}\lambda(i)_{j}e_{j}^{\prime
}~~~~\ (i=1,\ldots,d). \label{e4}%
\end{align}

The equations (\ref{e1}) have no unknowns, so are solvable if and only if we
are in the good case. In that case, (\ref{e2}) says that $H$ and $Z$ are
invertible matrices, and (\ref{e3}) asserts that conjugation by $H$ effects
the isomorphism represented by $Z$ from $\overline{L}$ to $\overline
{L^{\prime}}$, and hence that it maps $\overline{N\beta}$ onto $\overline
{N^{\dag}\beta^{\prime}}$.

(\ref{e4}) says that $\Xi(i)$ lies in $\overline{N^{\dag}\beta^{\prime}}$ .

It is clear that a local solution to $\mathcal{F}(\theta)$ gives rise to an
isomorphism $\theta_{1}:\overline{N}\rightarrow\overline{N^{\dag}}$ and
elements $v_{1},\ldots,v_{d}\in\overline{N^{\dag}}$ having all the required
properties, and conversely that such a system $(\theta_{1};$ $v_{1}%
,\ldots,v_{d})$ gives rise to a local solution of $\mathcal{F}(\theta)$.

\section{Finding lattice-splittable subgroups 2: profinite
invariance\label{lastbit}}

For a virtually polycyclic group $G$ let $\Gamma(G)$ denote the $\mathcal{T}%
$-residual of $G$, that is the smallest normal subgroup $K$ such that $G/K$ is
torsion-free and nilpotent. Thus%
\[
\frac{\Gamma(G)}{\gamma_{n}(G)}=\frac{\tau_{n}(G)}{\gamma_{n}(G)}:=\tau\left(
\frac{G}{\gamma_{n}(G)}\right)
\]
for any sufficiently large $n$. It follows easily that $\Gamma$ is a
$\mathcal{P}$-invariant subgroup functor. The key result is

\begin{proposition}
\label{gammasupp}There is a $\mathcal{P}$-invariant $G\longmapsto
q(G)\in\mathbb{N}$ such that if the polycyclic group $G$ is nilpotent by
abelian and $\Gamma(G)\in\mathcal{T}$ then
\begin{equation}
\Gamma(G)G^{q(G)}=\Gamma(G)C \label{gamsup}%
\end{equation}
for some nilpotent subgroup $C$.
\end{proposition}

This is proved below. To deduce Proposition \ref{newfirst}, we follow the
procedure of \S \ref{finding}.

The aim is to find a $\mathcal{P}$-invariant subgroup functor that maps each
virtually polycyclic group to a characteristic lattice-splittable subgroup of
finite index. Step 1 of \S \ref{finding} defines a $\mathcal{P}$-invariant
subgroup functor $G\rightarrow G_{1}$ such that $G_{1}$ is a nice
characteristic subgroup of finite index in $G$ and $N:=\mathrm{Fit}(G_{1})$ is
a lattice group. Then $\Gamma(G_{1})\leq N$ and $G_{1}$ satisfies the
hypotheses of Proposition \ref{gammasupp}. Thus for any group $G^{\dag}$ in
the genus of $G$ with $N^{\dag}=\mathrm{Fit}(G^{\dag})$ there exists a
nilpotent subgroup $C^{\dag}$ such that%
\[
G_{2}^{\dag}:=N^{\dag}G_{1}^{\dag q(G_{1})}=N^{\dag}C^{\dag}.
\]
Then for each $e\in\mathbb{N}$ we have $N^{\dag}G_{2}^{\dag e}=N^{\dag}C^{\dag
e}$, so $\mathcal{X}(N^{\dag}G_{2}^{\dag e},N^{\dag})$ is non-empty. Now Lemma
\ref{X*conj} shows that for each $e$,
\begin{equation}
\mathcal{X}^{\ast}(N^{\dag}G_{2}^{\dag e},N^{\dag})\neq\varnothing
\Longleftrightarrow\mathcal{X}^{\ast}(NG_{2}^{e},N)\neq\varnothing.
\label{xstarthing}%
\end{equation}
Define $e(G)$ to be the least positive integer such that $\mathcal{X}^{\ast
}(NG_{2}^{e},N)\neq\varnothing$ - this exists because $(C^{e})_{s}^{\ast}$
fixes $N$ for some $e\in\mathbb{N}$ (\cite{S1} Chapter 7, Lemma 5). Then
(\ref{xstarthing}) shows (i) that $e(G)$ is a $\mathcal{P}$-invariant of $G$,
and (ii) that $N^{\dag}G_{2}^{\dag e(G)}$ is lattice-splittable for every
$G^{\dag}$ in the genus of $G$.

So finally we may define%
\[
G_{0}:=NG_{1}^{q(G_{1})e(G)}.
\]

This completes the proof of Proposition \ref{newfirst}, modulo Proposition
\ref{gammasupp}.

\medskip

The rest of this section is devoted to the proof of the latter.

For a group $G$ and a $G$-module $E$ define%
\[
\mu_{G}(E):=\gcd\left\{  \mathrm{ex}(E/E(x-1)\mid x\in G,~[G,x]\subseteq
\mathrm{C}_{G}(E)\right\}
\]
where for $Y\subseteq\mathbb{N}$, $\gcd(Y\cup\{\infty\}):=\gcd Y$ if
$Y\neq\varnothing$, $\infty$ if $Y=\varnothing$, and $\mathrm{ex}(A)$ denotes
the exponent of a group $A$.

We begin with

\begin{lemma}
\label{littlelemma}Let $H$ be a finitely generated nilpotent group and
$M\cong\mathbb{Z}^{n}$ an $H$-module. Assume that $H/K$ is abelian where
$K=\mathrm{C}_{H}(M)$ and that $\mu_{H}(M)$ is finite. Make $E=\widehat{M}$
into a module for $P=\widehat{H/K}$ in the natural way. Then

\begin{description}
\item[(i)] $\mu_{P}(E)=\mu_{H}(M)$;

\item[(ii)] $\mathrm{ex}H^{i}(H,M)\mid\mu_{H}(M)$ for each $i\geq1$.
\end{description}
\end{lemma}

\begin{proof}
We may assume that $K=1$, so $H\leq P$. As $\mu_{H}(M)$ is finite, there
exists $x\in H$ such that $m(x):=\mathrm{ex}(M/M(x-1))$ is finite. More
generally, let $x$ be any element of $P$ with $m:=m(x)$ finite. Write
$P(k)=\mathrm{C}_{P}(E/Ek)$ ($k\in\mathbb{N}$). For each prime $p$ there
exists $y_{p}\in G$ with $y_{p}\equiv x~(\operatorname{mod}P(mp))$. Write
$\eta_{p}=y_{p}-1$. Then $E\eta_{p}+Emp=E(x-1)+Emp\geq Em.$ This implies that
$(E\eta_{p}+Em)/E\eta_{p}$ is $p$-divisible.

If $\mathrm{C}_{M}(y_{p})>1$ then $M/M\eta_{p}$ maps onto $\mathbb{Z}$, so
$E/E\eta_{p}$ maps onto $\widehat{\mathbb{Z}}$, and then so does $(E\eta
_{p}+Em)/E\eta_{p}$. But $\widehat{\mathbb{Z}}$ is not $p$-divisible.
Therefore $\mathrm{C}_{M}(y_{p})=1$, and so $M/M\eta_{p}$ is finite. It
follows that $M/M\eta_{p}\cong E/E\eta_{p}$ and hence that $\mathrm{ex}%
(M/M\eta_{p})\mid ml_{p}$ where $p\nmid l_{p}$.

As $p$ was arbitrary this shows that that $\mu_{H}(M)\mid m(x)$, and letting
$x$ vary we infer that $\mu_{H}(M)\mid\mu_{P}(E)$.

It is easy to see that $\mu_{P}(E)\mid\mu_{H}(M)$, and \textbf{(i) }follows.

For \textbf{(ii)}, write $H^{i}(-):=H^{i}(H,-)$. Let $x\in H$ have $m(x)$
finite, and put $\xi=x-1$. According to \cite{LR}, \textbf{10.3.3}, the map
$\xi_{\ast}:H^{i}(M)\rightarrow H^{i}(M)$ induced by $\xi$ is zero, so the
cohomology sequence arising from the exact sequence $0\rightarrow
M\overset{\xi}{\rightarrow}M\rightarrow M/M\xi\rightarrow0$ induces an
injection $H^{i}(M)\rightarrow H^{i}(M/M\xi)$. It follows that%
\[
\mathrm{ex}H^{i}(M)\mid\mathrm{ex}H^{i}\left(  M/M\xi\right)  \mid
\mathrm{ex}\left(  M/M\xi\right)  =m(x).
\]
and this implies \textbf{(ii)}.
\end{proof}

\medskip

\emph{Remark.} Essentially the same argument is at the heart of the proof of
Proposition \ref{compconj}.

\medskip

We shall use the fact that if $N$ is a nilpotent group of class at most $c$
then%
\begin{equation}
(N^{e})^{f}\supseteq N^{q}\text{ where }q=(ef)^{c}; \label{nilprod}%
\end{equation}
this follows e.g. from \cite{S1} Chapter 6, Proposition 2.

The $\mathcal{P}$-invariance of some of the following steps is a routine
application of \S \ref{proftech1}; it will be explicitly justified where it is
not obvious.

\bigskip

\bigskip\textbf{Proof of Proposition \ref{gammasupp}. }Let $G$ be nilpotent by
abelian with $\Gamma:=\Gamma(G)\in\mathcal{T}$, and let $c$ denote the
nilpotency class of $G/\Gamma$.

If $G$ is virtually nilpotent, define $q(G)$ to be the least $q\in\mathbb{N}$
such that $G^{q}$ is nilpotent.

Suppose henceforth that $G$ is not virtually nilpotent. Then $\Gamma$ is
infinite, and we argue by induction on the Hirsch length of $\Gamma$. Note
that $\Gamma\leq\mathrm{Fit}(G):=N$ and that $G/N$ is abelian. Put $T=\tau
_{2}(\Gamma)$. Then $\Gamma/T$ is infinite, and $\mathrm{Fit}(G/T)=N/T$ by
\cite{S1} Chapter 3 Proposition 3.

Set $N_{i}/T=\tau_{i}(N/T)$. Then $N=N_{1}>N_{2}$ and there exists $i\geq2$
such that $N_{i-1}\geq\Gamma>\Gamma\cap N_{i}:=Z$. Put $A:=\Gamma/Z$; then
$A\cong\mathbb{Z}^{n}$ for some $n$, and $N$ centralizes $A$.

\medskip

\emph{Claim:} if $x\in G$ acts nilpotently on $A$ then $x\in N$. For $x$ then
acts nilpotently on $N/N_{2}$, so $N\left\langle x\right\rangle /T$ is
nilpotent, by another application of \cite{S1} Chapter 3 Proposition 3; but
$N\left\langle x\right\rangle \vartriangleleft G$ since $G/N$ is abelian, so
$N\left\langle x\right\rangle /T\leq\mathrm{Fit}(G/T)=N/T$.

\medskip

Now let $x\in G\smallsetminus N$. Denote by $\xi$ the action of $x-1$ on $A$,
and write $^{\ast}\xi^{j}$ for the kernel of $\xi^{j}$. Then%
\[
\Gamma/Z=A>\,^{\ast}\xi^{n}=\,^{\ast}\xi^{n+1}:=\Delta_{x}/Z
\]
say. Note that $\Delta_{x}\vartriangleleft G$ because $G/N$ is abelian.

\medskip

\emph{Case 1:} Where $\Delta_{x}=1$ for some $x\in G\smallsetminus N$. Then
$\Gamma=A$ and\ $C_{A}(x)=0$. It follows that $m:=\mu_{H}(A)$ is finite, where
$H=G/\Gamma$. Then $m.H^{2}(H,A)=0$ by Lemma \ref{littlelemma} (ii). The proof
of \cite{LR} \textbf{10.1.15 } shows that then $AG^{f}$ is a split extension
of $A$ where $f=m^{n}$. Thus $G=\Gamma C$ where $C\cong G/\Gamma$ is
nilpotent. Set $q(G):=f$. That this is $\mathcal{P}$-invariant is not quite
straightforward, see below.

\medskip

\emph{Case 2: }Suppose that $\Delta_{x}\neq1$ for every $x\in G\smallsetminus
N$. For each such $x$ the group $\Delta_{x}$ is infinite (as $\Gamma$ is
torsion-free), and inductively we may suppose that $q_{x}:=q(G/\Delta_{x})$
has been defined. We choose $x$ so as to minimize $q_{x}$, and set
$q_{1}:=q_{1}(G):=q_{x}$. Thus $G$ has a subgroup $D\geq\Delta_{x}$ such that
$D/\Delta_{x}$ is nilpotent and $\Gamma G^{q_{x}}=\Gamma D$. If $D$ is maximal
such, we have
\[
D/\Delta_{x}\in\mathcal{X}(G_{1}/\Delta_{x},\Gamma/\Delta_{x})
\]
where $G_{1}=\Gamma G^{q_{1}}$. Also $\Gamma(D)\leq\Delta_{x}$ which has
infinite index in $\Gamma$; so $q(D)$ is also defined, and%
\[
\Gamma(D)D^{q(D)}=\Gamma(D)C
\]
for some nilpotent subgroup $C$. Now among all such maximal groups $D$, choose
$D$ with $q(D)$ as small as possible, and set $q_{2}:=q_{2}(G):=q(D)$.

Putting $q=(q_{1}q_{2})^{c}$ we get $\Gamma G^{q}=\Gamma C$. Again, the
question of $\mathcal{P}$-invariance remains to be discussed; we do this below.

\medskip

$\mathcal{P}$-\emph{invariance:} Suppose that $G^{\dag}$ satisfies the same
hypotheses as $G,$ and that $\widehat{G^{\dag}}\cong\widehat{G}$. If $G$ is
virtually nilpotent then so is $G^{\dag}$, and it is easy to see that
$q(G^{\dag})=q(G)$ in this case.

Otherwise, define $T^{\dag}\leq Z^{\dag}<\Gamma^{\dag}\leq N^{\dag}$ by
analogy to $T,~Z$ etc. These are all $\mathcal{P}$-invariants. Corollary
\ref{neatcor} ensures that there is an isomorphism $\Theta:\widehat{G}%
\rightarrow\widehat{G^{\dag}}$ such that $(\overline{N}G)\Theta=\overline
{N^{\dag}}G^{\dag}$.

\medskip

\emph{Case 1bis:} Suppose that $\Delta_{x}=1$ for some $x\in G\smallsetminus
N$. This implies that $Z=1=C_{\Gamma}(x)$. Let $y\in G^{\dag}$ satisfy
$y\equiv x\Theta~(\operatorname{mod}\overline{N^{\dag}})$. Then $\overline
{Z^{\dag}}=Z\Theta=1$ and $C_{\overline{\Gamma^{\dag}}}(y)=C_{\overline
{\Gamma}}(x)\Theta=1$, which in turn implies that $\Delta_{y}^{\dag}=1$ in the
obvious notation. Writing $A^{\dag}=\Gamma^{\dag}$, Lemma \ref{littlelemma}
(i) now shows that%
\[
\mu_{G/\Gamma}(A)=\mu_{\widehat{G}/\overline{N}}(\widehat{A})=\mu
_{\widehat{G^{\dag}}/\overline{N^{\dag}}}(\widehat{A^{\dag}})=\mu_{G^{\dag
}/\Gamma^{\dag}}(A^{\dag}),
\]
as $\Theta$ transfers the $\widehat{G}/\overline{N}$-module $\widehat{A}$ into
the $\widehat{G^{\dag}}/\overline{N^{\dag}}$-module $\widehat{A^{\dag}}$ in
the obvious way.

Thus if $G$ falls into \emph{Case }1, then $q(G)$ as there defined is indeed
$\mathcal{P}$-invariant.

\medskip

\emph{Case 2bis:} Suppose that $G$ falls into \emph{Case 2}. Then so does
$G^{\dag}$ by symmetry and the preceding argument. We chose $x\in
G\smallsetminus N\,$ so as to minimize $q_{1}:=q_{1}(G):=q(G/\Delta_{x})$. Now
$\Delta_{x}/Z$ is the $n$th iterated centralizer of $x$ in the module
$\Gamma/Z$, so taking $y\equiv x\Theta$ as above we infer in a similar way
that $\overline{\Delta_{y}^{\dag}}/\overline{Z^{\dag}}$ is the $n$th iterated
centralizer of $y$ in the module $\overline{\Gamma^{\dag}}/\overline{Z^{\dag}%
},$ which implies that $\overline{\Delta_{x}}\Theta=\overline{\Delta_{y}%
^{\dag}}$.

Thus $\Theta$ induces isomorphisms $\Theta_{1}:\widehat{G/\Delta_{x}%
}\rightarrow\widehat{G^{\dag}/\Delta_{y}^{\dag}}$ and $\Theta_{2}%
:\widehat{\Delta_{x}}\rightarrow\widehat{\Delta_{y}^{\dag}}$. The first
implies that $q_{1}=q(G/\Delta_{x})=q(G^{\dag}/\Delta_{y}^{\dag}):=q_{y}$. If
$y^{\prime}\in$ $G^{\dag}\smallsetminus N^{\dag}$ is arbitrary, a symmetry
argument shows that $q(G^{\dag}/\Delta_{y^{\prime}}^{\dag})=q(G/\Delta
_{x^{\prime}})$ for some $x^{\prime}\in G\smallsetminus N$, whence $q(G^{\dag
}/\Delta_{y^{\prime}}^{\dag})\geq q_{y}$. So $q_{1}(G^{\dag})=q_{y}=q_{1}$.
Thus $q_{1}$ is a $\mathcal{P}$-invariant.

Put $G_{1}^{\dag}=\Gamma^{\dag}G^{\dag q_{1}}$. As $q_{1}=q_{y},$ the set
$\mathcal{X}(G_{1}^{\dag}/\Delta_{y}^{\dag},\Gamma^{\dag}/\Delta_{y}^{\dag})$
is non-empty, while $D/\Delta_{x}$ lies in $\mathcal{X}(G_{1}/\Delta
_{x},\Gamma/\Delta_{x})$. Then $\overline{D}/\overline{\Delta_{x}}%
\in\mathcal{X}(\overline{G_{1}}/\overline{\Delta_{x}},\overline{\Gamma
}/\overline{\Delta_{x}})$ by Lemma \ref{maxymaxy}, whence $(\overline
{D}/\overline{\Delta_{x}})\Theta_{1}\in\mathcal{X}(\overline{G_{1}^{\dag}%
}/\overline{\Delta_{y}^{\dag}},\overline{\Gamma^{\dag}}/\overline{\Delta
_{y}^{\dag}})$. Now Corollary \ref{finconjmaxns} implies that $(\overline
{D}/\overline{\Delta_{x}})\Theta_{1}$ is conjugate in $\overline{G_{1}^{\dag}%
}/\overline{\Delta_{y}^{\dag}}$ to $\overline{D_{1}}/\overline{\Delta
_{y}^{\dag}}$ for some $D_{1}/\Delta_{y}^{\dag}\in\mathcal{X}(G_{1}^{\dag
}/\Delta_{y}^{\dag},\Gamma^{\dag}/\Delta_{y}^{\dag})$. Then $\overline
{D}\Theta$ is conjugate in $\overline{G_{1}^{\dag}}$ to $\overline{D_{1}},$ so
$\widehat{D}\cong\widehat{D_{1}}$, whence $q(D_{1})=q(D)$.

Recall that we chose $D$ so as to minimize $q(D)$. A symmetry argument similar
to the one used above shows that then $q(D)$ is a $\mathcal{P}$-invariant of
$G$, and we define $q_{2}(G):=q(D).$

Setting $q:=q(G):=(q_{1}(G)q_{2}(G))^{c}$ we see as above that $\Gamma^{\dag
}G^{\dag q}=\Gamma^{\dag}C^{\dag}$ for some nilpotent subgroup $C^{\dag}$, and
we are done.

\bigskip\emph{All Souls College, Oxford}
\end{document}